\documentclass[12pt, leqno]{amsart}

\usepackage{graphicx}
\usepackage{amssymb,amsmath,amsthm,amscd}
\usepackage{mathrsfs,mathdots}
\usepackage{enumerate,stmaryrd}
\usepackage{upgreek}
\usepackage[usenames,dvipsnames]{color}
\usepackage[colorlinks=true, pdfstartview=FitV,
 linkcolor=blue,citecolor=blue,urlcolor=blue]{hyperref}
\usepackage[all]{xy}

\usepackage{verbatim}
\usepackage{oldgerm}
\usepackage[normalem]{ulem}  
\usepackage{xspace}
\usepackage{xcolor,cancel}
\usepackage{tikz}
\tikzset{dynkdot/.style={circle,draw,scale=.38}}

\usepackage{marginnote}

\numberwithin{equation}{section}
\allowdisplaybreaks[4]

\usepackage[colorinlistoftodos]{todonotes}
\newcommand{\nc}{\newcommand}

\renewcommand{\le}{\leqslant}

\renewcommand{\ge}{\geqslant}

\nc{\op}{\operatorname}

\theoremstyle{plain}
\newtheorem{lemma}{Lemma}[section]
\newtheorem{proposition}[lemma]{Proposition}
\newtheorem{theorem}[lemma]{Theorem}
\newtheorem*{maintheorem}{Main Theorem}

\theoremstyle{definition}
\newtheorem{remark}[lemma]{Remark}

\newtheorem{example}[lemma]{Example}

\newtheorem{definition}[lemma]{Definition}
\newtheorem{corollary}[lemma]{Corollary}
\newtheorem{convention}[lemma]{Convention}

\nc{\Prop}{\begin{proposition}}
\nc{\enprop}{\end{proposition}}
\nc{\Lemma}{\begin{lemma}}
\nc{\enlemma}{\end{lemma}}
\nc{\Cor}{\begin{corollary}}
\nc{\encor}{\end{corollary}}
\nc{\Rem}{\begin{remark}}
\nc{\enrem}{\end{remark}}
\nc{\Def}{\begin{definition}}
\nc{\edf}{\end{definition}}
\nc{\shc}{\mathcal{C}}

\newcommand{\Q}{\mathbb{Q}}
\newcommand{\C}{\mathbb{C}}

\newcommand{\Seq}{\Sigma}

\newcommand{\dT}{\mathrm{T}}

\newcommand{\Ca}{\mathscr{C}}
\nc{\F}{\mathcal{F}}
\newcommand{\D}{\mathscr{D}}
\nc{\HOM}{\on{H\textsc{om}}}
\newcommand{\M}{\mathrm{M}}
\newcommand{\W}{\mathsf{W}}

\newcommand{\Z}{\mathbb{\ms{1mu}Z}}
\newcommand{\seteq}{\mathbin{:=}}
\newcommand{\hd}{{\operatorname{hd}}}
\newcommand{\soc}{{\operatorname{soc}}}
\nc{\ov}[1]{\overline{#1}}
\nc{\Wlmj}[3]{\W_{#2,#3}^{(#1)}}
\nc{\Mkl}[2]{\M_\ttww(#1,#2)}
\nc{\rmat}[1]{{\mathbf r}_{\mspace{-2mu}\raisebox{-.5ex}{${\scriptstyle{#1}}$}}}
\newcommand{\on}{\operatorname}
\nc{\de}{\on{\textfrak{d}}}
\nc{\tL}{\widetilde{\Lambda}}
\nc{\tl}{\widetilde{\lambda}}
\nc{\mqs}{(-q^2)}
\nc{\Cquiver}{\upsigma}

\nc{\mut}[1]{{\mu}_{\mspace{-2mu}\raisebox{-.5ex}{${\scriptstyle{#1}}$}}}

\newcommand{\g}{{\mathfrak{g}}}
\newcommand{\h}{{\mathfrak{h}}}
\newcommand{\n}{{\mathfrak{n}}}
\newcommand{\isoto}[1][]{\mathop{\xrightarrow%
[{\raisebox{.3ex}[0ex][.3ex]{$\scriptstyle{#1}$}}]%
{{\raisebox{-.6ex}[0ex][-.6ex]{$\mspace{2mu}\sim\mspace{2mu}$}}}}}
\newcommand{\longisoto}[1][]{\mathop{\xrightarrow%
[{\raisebox{.3ex}[0ex][.3ex]{$\scriptstyle{\hs{1ex}#1\hs{1ex}}$}}]%
{{\raisebox{-.6ex}[0ex][-.6ex]{$\hs{1ex}\sim\hs{1ex}$}}}}}

\newcommand{\id}{\on{id}}

\newcommand{\soplus}{\mathop{\mbox{\normalsize$\bigoplus$}}\limits}

\newcommand{\sotimes}{\mathop{\mbox{\normalsize$\bigotimes$}}\limits}

\newcommand{\ww}{ \textbf{\textit{w}}}
\newcommand{\ttww}{{\widetilde{\ww}} }

\nc{\nconv}{\mathop{\mbox{\large $\odot$}}}
\nc{\nnconv}{\mathop{\mbox{\large $\star$}}}

\nc{\lb}{\llbracket}
\nc{\rb}{\rrbracket}

\newcommand{\ko}{{{\mathbf{k}}}}
\nc{\la}{\lambda}
\nc{\La}{\Lambda}

\nc{\tLa}{\widetilde{\Lambda}}
\nc{\ve}{\varepsilon}
\nc{\ep}{\epsilon}
\nc{\vp}{\varphi}
\nc{\lan}{\langle}
\nc{\ran}{\rangle}
\nc{\Uqg}{U_q(\g)}
\nc{\Aqg}{A_q(\g)}
\nc{\Aqn}{A_q(\n)}
\nc{\al}{\alpha}
\nc{\be}{\beta}
\nc{\ga}{\gamma}
\nc{\wt}{\operatorname{wt}}
\nc{\ch}{\operatorname{ch}}

\nc{\norm}{{\mathrm{norm}}}
\nc{\aff}{{\mathrm{aff}}}
\nc{\Maf}{M_\aff}
\nc{\ev}{{\mathrm{even}}}
\nc{\od}{{\mathrm{odd}}}
\nc{\Sev}{\Seq^{\ev}}
\nc{\Sod}{\Seq^{\od}}
\nc{\Spl}{\Seq^{+}}
\nc{\Smi}{\Seq^{-}}
\nc{\low}{{\mathrm{low}}}
\nc{\upper}{{\mathrm{up}}}
\nc{\one}{{\bf{1}}}
\nc{\To}[1][{\hspace{2ex}}]{\xrightarrow{\,#1\,}}
\nc{\te}{\tilde{e}}
\nc{\tw}{{\widetilde{w}}}
\nc{\tww}{\ww}
\nc{\tuu}{{\mathsf{u}}}
\nc{\tel}{\tilde{e}^\low}
\nc{\teu}{\tilde{e}^\upper}
\nc{\tf}{\tilde{f}}
\nc{\tfl}{\tilde{f}^\low}
\nc{\tfu}{\tilde{f}^\upper}
\nc{\tE}{\widetilde{E}}
\nc{\tF}{\widetilde{F}}
\nc{\tFF}{\widetilde{\F}}
\nc{\tB}{\widetilde{B}}
\nc{\tz}{\tilde{z}}
\nc{\tQ}{\hspace{-.2ex}\textbf{\textit{Q}}}
\nc{\tb}{\tilde{b}}
\nc{\Ft}{\F^\dT}
\nc{\Seed}{\mathcal{S}}
\newenvironment{rouge}
{\color{red}}
{}

\nc{\er}{\end{rouge}}

\newcommand{\berm}{\begin{rouge}{}\marginnote{\fbox{\scshape\lowercase{M}}}
{}}

\newcommand{\bero}{\begin{rouge}{}\marginnote{\fbox{\scshape\lowercase{O}}}
{}}

\newenvironment{blue}
{\color{Dandelion}}
{}

\nc{\beb}{\begin{blue}}
\nc{\eb}{\end{blue}}

\newenvironment{bluk}
{\relax\color{blue}}
{\hspace*{.5ex}\relax}

\newcommand{\bek}{\begin{bluk}}
\newcommand{\ek}{\end{bluk}}

\nc{\cor}{\mathbf{k}}
\nc{\tens}{\mathop\otimes}
\nc{\gmod}{\mbox{-$\mathrm{gmod}$}}
\nc{\md}{\mbox{-$\mathrm{mod}$}}
\nc{\uqm}{\mathscr{C}_\g}
\nc{\gMod}{\mbox{-$\mathrm{gMod}$}}

\nc{\proj}{\mbox{-$\mathrm{proj}$}}
\nc{\gproj}{\mbox{-$\mathrm{gproj}$}}
\nc{\smod}{\mbox{-$\mathrm{mod}$}}
\nc{\nmod}{\mbox{-$\mathrm{nilmod}$}}

\nc{\seed}{\mathscr{S}}

\newcommand{\cmA}{\mathsf{A}}

\nc{\Rnorm}{R^{\mathrm{norm}}}
\nc{\Runiv}{R^{\ms{1mu}\mathrm{univ}}}
\nc{\Rren}{R^{\ms{1mu}\mathrm{ren}}} 
\nc{\col}{\colon}
\nc{\epiTo}[1][]{\xymatrix{\ar@{->>}[r]^-{{#1}}&}}
\nc{\epito}{\twoheadrightarrow}
\nc{\monoTo}[1][]{\xymatrix{\ar@{>->}[r]^-{{#1}}&}}
\nc{\monogets}[1][]{\xymatrix{&\ar@{_{(}->}[l]^-{{#1}}}}
\nc{\sym}{\mathfrak{S}}
\nc{\rl}{\mathsf{Q}}
\nc{\prl}{\rl_+}
\nc{\crl}{\mathsf{Q}^\vee}
\nc{\pcrl}{\crl_+}
\nc{\Qq}{{\Q(q)}}
\nc{\wl}{\mathsf{P}}   
\nc{\Oint}{\mathcal{O}_{{\mathrm{int}}}}
\newcommand{\scbul}{{\,\raise1pt\hbox{$\scriptscriptstyle\bullet$}\,}}

\nc{\conv}{\mathop{\mathbin{\mbox{\large $\circ$}}}}
\newcommand{\hconv}{\mathbin{\scalebox{.9}{$\nabla$}}}
\newcommand{\sconv}{\mathbin{\scalebox{.9}{$\Delta$}}}

\nc{\pv}{  \to\updownarrow\gets }
\nc{\nv}{  \longleftrightarrow {\raise -1pt\hbox{$\hspace{-2ex}\begin{matrix}\downarrow \\[-1ex] \uparrow\end{matrix}$}} }

\newcommand{\Hom}{\operatorname{Hom}}

\renewcommand{\Im}{\op{Im}}

\newcommand{\ex}{{\mathrm{ex}}}
\nc{\K}{{K}}
\nc{\Kex}{{\K}^{\mathrm{ex}}}
\nc{\Uex}{\Uppsi_{\mathrm{ex}}}
\nc{\Kfr}{{\K}^{\mathrm{f\mspace{.01mu}r}}}
\nc{\cl}{{\mathrm{cl}\ms{1mu}}}
\nc{\ben}{\begin{enumerate}}
\nc{\ee}{\end{enumerate}}
\nc{\bnum}{\begin{enumerate}[{\rm(i)}]}
\nc{\bna}{\begin{enumerate}[{\rm(a)}]}
\nc{\bc}{\begin{cases}}
\nc{\ec}{\end{cases}}

\newenvironment{myequation}
{\relax\setlength{\arraycolsep}{1pt}\begin{eqnarray}}
{\end{eqnarray}}
\newenvironment{myequationn}
{\relax\setlength{\arraycolsep}{1pt}\begin{eqnarray*}}
{\end{eqnarray*}}

\nc{\eq}{\begin{myequation}}
\nc{\eneq}{\end{myequation}}
\nc{\eqn}{\begin{myequationn}}
\nc{\eneqn}{\end{myequationn}}
\nc{\hs}{\hspace*}

\newenvironment{myarray}[1]{\relax\setlength{\arraycolsep}{.5pt}

\begin{array}{#1}}{\end{array}\relax}

\newcommand{\ba}{\begin{myarray}}
\newcommand{\ea}{\end{myarray}}

\nc{\noi}{\noindent}
\nc{\ang}[1]{\langle{#1}\rangle}
\nc{\fr}{{\mathrm{fr}}}
\nc{\qt}[1]{\quad\text{#1}\quad}
\nc{\ol}{\overline}
\nc{\true}{\delta}
\nc{\ms}{\mspace}
\renewcommand{\mod}{\ms{3mu}\mathbin{\mathrm{mod}}\ms{1mu}}
\nc{\vs}{\vspace*}
\nc{\bl}{\bigl(}
\nc{\br}{\bigr)}
\nc{\bep}{\ol{\ep}}
\nc{\bal}{\,\ol{\al}}
\nc{\qtq}[1][{and}]{\quad\text{#1}\quad}
\nc{\set}[2]{\left\{{#1}\mid{#2}\right\}}
\nc{\ro}{{\rm(}}
\nc{\rf}{{\rm)}\xspace}
\nc{\Proof}{\begin{proof}}
\nc{\QED}{\end{proof}}
\nc{\monoto}[1][]{\xymatrix@C=2ex{\ar@{>->}[r]^-{{#1}}&}\ms{-8mu}}
\nc{\etens}{\boxtimes}
\nc{\height}[1]{\vert{#1}\vert}
\nc{\dsum}{\mathop\sum\limits}
\nc{\Lrev}{L^{\bek{\rev}\ek}}
\nc{\rev}{\mathrm{rev}}
\nc{\fw}{\Lambda}
\nc{\uqpg}{U_q'(\g)}
\nc{\tp}{\ms{1.5mu}{\widetilde{p}}\ms{2mu}}
\nc{\Deg}{\mathrm{Deg}}
\nc{\Bg}{\mathcal{G}}

\nc{\wb}[1]{\mbox{$\rule[-1.1ex]{0ex}{2ex}#1$}}

\nc{\bwr}{\mbox{\large$\wr$}}
\nc{\vphi}{\varphi}
\nc{\G}{\mathcal{G}}
\nc{\tD}{\widetilde{\mathrm{De}}\mathrm{g}}
\nc{\Li}{\La^\infty}
\nc{\Di}{\Deg^\infty}
\nc{\zero}{\ms{1mu}\mathrm{zero}\ms{1mu}}
\nc{\cwl}{\wl^\vee}
\nc{\rc}{renormalizing coefficient\xspace}
\nc{\cz}{{\cor[z^{\pm1}]}}
\nc{\ake}[1][2ex]{\rule[-.5ex]{0ex}{#1}}
\nc{\akew}[1][2ex]{\rule[-1ex]{#1}{0ex}}
\nc{\rd}{{}^*\ms{-3mu}}
\nc{\st}[1]{\{{#1}\}}
\nc{\corh}{\widehat{\cor}}
\nc{\czt}{\cz^\times}
\nc{\eps}{\varepsilon}
\nc{\rr}{rationally renormalizable\xspace}
\nc{\QHA}{\mathrm{QHA}}

\newenvironment{magem}{\relax\color{magenta}}{\relax}

\newcommand{\bema}{\begin{magem}}
\newcommand{\ema}{\end{magem}}

\newlength{\mylength}
\title[Monoidal categorification and quantum affine algebras]{Monoidal categorification and quantum affine algebras}

\author[M. Kashiwara]{Masaki Kashiwara}
\thanks{The research of M.\ Kashiwara
was supported by Grant-in-Aid for Scientific Research (B)
15H03608, Japan Society for the Promotion of Science.}
\address[M. Kashiwara]{
Kyoto University Institute for Advanced Study,
Research Institute for Mathematical Sciences, Kyoto University,
Kyoto 606-8502, Japan \& Korea Institute for Advanced Study, Seoul 02455, Korea }
\email[M. Kashiwara]{masaki@kurims.kyoto-u.ac.jp}

\author[M. Kim]{Myungho Kim}
\address[M. Kim]{Department of Mathematics, Kyung Hee University, Seoul 02447, Korea}
\email[M. Kim]{mkim@khu.ac.kr}
\thanks{The research of M.\ Kim was supported by the National Research Foundation of
Korea(NRF) Grant funded by the Korea government(MSIP) (NRF-2017R1C1B2007824).}

\author[S.-j. Oh]{Se-jin Oh}
\thanks{ The research of S.-j.\ Oh was supported by the Ministry of Education of the Republic of Korea and the National Research Foundation of Korea (NRF-2019R1A2C4069647).}
\address[S.-j. Oh]{Department of Mathematics, Ewha Womans University, Seoul 03760, Korea}
\email[S.-j. Oh]{sejin092@gmail.com}

\author[E. Park]{Euiyong Park}
\thanks{The research of E.\ P.\ was supported by the National Research Foundation of Korea(NRF) Grant funded by the Korea Government(MSIP)(NRF-2017R1A1A1A05001058).}
\address[E. Park]{Department of Mathematics, University of Seoul, Seoul 02504, Korea}
\email[E. Park]{epark@uos.ac.kr}



\date{\today}

\begin{document}

\begin{abstract}
We introduce and investigate new invariants on the pair of modules
$M$ and $N$ over quantum affine algebras $U_q'(\g)$ by analyzing
their associated $R$-matrices. From new invariants, we provide a
criterion for a monoidal category of finite-dimensional integrable
$U_q'(\g)$-modules to become a monoidal categorification of a
cluster algebra.
\end{abstract}

\maketitle
\tableofcontents

\section{Introduction}

For an affine Kac-Moody algebra $\g$, let $U_q'(\g)$ be the corresponding quantum affine algebra.
Since the category $\Ca_\g$ of finite-dimensional integrable representations over
$U_q'(\g)$ has a rich structure including rigidity, it has been intensively studied in various fields of mathematics and physics
(see \cite{AK,CP94,FR99,GV93,KS95,Nak01} for examples). In particular,  the representation theory for $\Ca_{\widehat{\mathfrak{sl}}_2}$ is well-understood:
every simple module in $\Ca_{\widehat{\mathfrak{sl}}_2}$ is isomorphic to a tensor product $S_1 \otimes S_2 \otimes \cdots \otimes S_r$ of simple modules,
called \emph{Kirillov-Reshetikhin} modules, satisfying that $S_i$ and $S_j$ \emph{are in general position} \cite{CP91}.
A simple object $S$ of a monoidal category is said to be \emph{real} if $S \otimes S$
is simple, and to be \emph{prime} if there exists no non-trivial factorization $S \simeq S_1 \otimes S_2$.
Since Kirillov-Reshetikhin modules over $U_q'(\widehat{\mathfrak{sl}}_2)$ are prime and real, every simple module in $\Ca_{\widehat{\mathfrak{sl}}_2}$
is real and can be expressed as a tensor monomial of prime real simple modules. However, these phenomena  cannot be expected for general $\g$.
In fact, there exist $\g$
 such that $\uqm$ contains  non-real simple modules~\cite{L03}.

The cluster algebras were introduced by Fomin and Zelevinsky in \cite{FZ02} for studying
the upper global bases of quantum groups \cite{K93,Lus93} and total positivity \cite{Lus94} in the viewpoint of combinatorics. Since their introduction,  numerous connections and applications have been discovered in
diverse fields of mathematics including representation theory, tropical geometry, integrable system and Poisson geometry (see \cite{BM06,FG06,FZ03,FN05,GSV}).

The representation theory of quantum affine algebras and the cluster algebras are connected by the notion of \emph{monoidal categorification},
introduced by  Hernandez-Leclerc in~\cite{HL10}.
A monoidal category $\shc$ is called a monoidal categorification of a cluster algebra $\mathscr{A}$ if it satisfies
\bna \item
the Grothendieck ring $K(\shc)$ of $\shc$ is isomorphic to $\mathscr{A}$ and
\item each cluster monomial of $\mathscr{A}$ corresponds to a real simple object in $\shc$, under the isomorphism. \ee
(This definition is weaker than the original one.)
Note that, by the \emph{Laurent phenomenon} of $\mathscr{A}$ (\cite{FZ02}), the \emph{Laurent positivity}  (proved by \cite{LS13} in a general setting) follows immediately,
if $\shc$ is a monoidal categorification of $\mathscr{A}$.

The notion of a monoidal categorification is extended in \cite{KKKO18} to quantum cluster algebras, a $q$-analogue of cluster algebras,
which were introduced by Berenstein and Zelevinsky in \cite{BZ05}.
Unlike cluster algebras, cluster variables are not commutative
but $q$-commutative, where the $q$-commutation relation is controlled by a skew-symmetric matrix $L$.
In \cite{GLS13},  Gei\ss, Leclerc and Schr{\"o}er showed that the quantum unipotent coordinate algebra $A_q(\mathfrak{n}(w))$, associated with a symmetric quantum group $U_q(\mathsf{g})$ and its Weyl
group element $w$, has a skew-symmetric
quantum cluster algebra structure  (see~\cite{GY16} for the non-symmetric case).
Using the quiver Hecke algebras $R$ introduced by Khovanov--Lauda \cite{KL09,KL11} and
Rouquier \cite{R08,R11} independently, the authors in \cite{KKKO18} introduced certain monoidal subcategory $\shc_w$ of the category  $R\gmod$ of finite-dimensional graded modules over $R$ and proved that $\shc_w$
gives a monoidal categorification for $A_q(\mathfrak{n}(w))$ in the following sense:
\begin{eqnarray} &&
\parbox{84ex}{
\begin{enumerate}[{\rm (i)}]
\item $\Z[q^{\pm 1/2}] \otimes_{\Z[q^{\pm1}]} K(\shc_w) \simeq A_{q^{1/2}}(\mathfrak{n}(w)) \seteq
\Z[q^{\pm 1/2}] \otimes_{\Z[q^{\pm1}]} A_q(\mathfrak{n}(w))$,
\item there exists a \emph{quantum monoidal seed} $\seed=(\{ V_i \}_{i \in K},L,\tB, D)$ in $\shc_w$, consisting of  a strongly commuting family $\{ V_i\}_{i\in\K}$ of real simple modules in $\shc_w$,
the $K\times K$ $\Z$-valued matrix $L=(-\Lambda(V_i,V_j))_{i,j\in K}$, an exchange matrix $\tB$ and a set $D$ of \emph{weights} of $V_i$'s in the root lattice of $\mathsf{g}$,
such that $[\seed]\seteq ( \{ q^{m_i}[V_i] \}_{i \in K}, L,\tB)$ is a quantum seed of $A_{q^{1/2}}(\mathfrak{n}(w))$ for some $m_i \in \tfrac{1}{2}\Z$,
\item $\seed$ \emph{admits successive mutations in all directions in $\Kex$}.
\end{enumerate}
}\label{eq: q-monoidal categorification}
\end{eqnarray}
Here $\Lambda(V,W)$ denotes the degree of the $R$-matrix $\rmat{V,W}$, constructed in \cite{KKK18A}, which is a distinguished homomorphism from $V \conv W$ to $W \conv V$,
where $V \conv W$ denotes the convolution product of $V$ and $W$
 (see \cite{KKKO18} for notations).
Note that the condition (iii) in~\eqref{eq: q-monoidal categorification} is
not easy to check since it is concerned with infinitely many mutations.
In the first part of \cite{KKKO18}, it was proved that the conditions (ii) and (iii) in~\eqref{eq: q-monoidal categorification}
are a consequence of the following condition:
\begin{enumerate}
\item[{\rm (ii$'$)}] there exists an \emph{admissible} monoidal seed $\seed=(\{ V_i \}_{i \in K},\tB)$ in $\shc_w$
such that $[\seed]\seteq ( \{ q^{m_i}[V_i] \}_{i \in K}, (-\Lambda(V_i,V_j))_{i,j\in K},\tB)$ is a quantum seed of $A_{q^{1/2}}(\mathfrak{n}(w))$ for some $m_i \in \tfrac{1}{2}\Z$ (see Definition~\ref{def:admissible}).
\end{enumerate}
Here, the admissibility means that the monoidal seed admits the first step mutations.
Thus, {\rm (ii$'$)} implies that, to achieve a monoidal categorification, it suffices to check the existence of such $M_k'$
only at the first mutation in each direction $k$.

On the whole flow of \cite{KKKO18}, the integer-valued invariants $\Lambda(V,W)$, $\tLa(V,W)$ and $\de(V,W)$, arising from the $\Z$-grading
structure of $R$ and defined in \cite{KKK18A,KKKO18}, provide important information in the representation theory of quiver Hecke algebra $R$.
To name a few,
(i) $\Lambda$ provides information about  whether the restriction of $R$-matrix $\rmat{V,W}$ to $V' \conv W'$ vanishes or not
for subquotients $V'$ and $W'$ of $V$ and $W$ respectively,
(ii) $\Lambda$  indicates the head and socle in the constituent of $V \conv W$,
(iii) $\tLa(V,W) \seteq \tfrac{1}{2}(\La(V,W)+(\wt(V),\wt(W)))$ measures the degree shifts of $V \conv W$ from the \emph{self-duality},
(iv) the non-negative integer $\de(V,W) \seteq  \tfrac{1}{2}(\La(V,W)+\La(W,V))$ tells whether $V \conv W$ is simple or not, corresponding to $\de(V,W)=0$ or $>0$,
under suitable assumptions on $V$ and $W$. However, in general, computing those values are quite difficult, and the second half of \cite{KKKO18}
is devoted to investigating several properties of those invariants.

 After the success in the quiver Hecke algebras setting, it is natural to ask a criterion for monoidal categorification for subcategories of  $\Ca_\g$.
There are monoidal subcategories $\Ca_N$ $(N \in \Z_{\ge1})$, $\Ca_\g^{-}$ and $\Ca_\g^{0}$ of $\Ca_\g$, introduced by Hernandez-Leclerc in \cite{HL10,HL16}
(see also \cite{KKOP19B} for $\Ca_\g^{0}$), whose Grothendieck rings $K(\shc)$ have cluster algebra structures, and which are conjectured to be monoidal categorifications of $K(\shc)$.
The conjecture for $\Ca_\g^{0}$ of affine types $A_n^{(t)}$ $(t=1,2)$ and $B_n^{(1)}$ is proved in \cite{KKOP19B} \emph{indirectly} by using
 generalized quantum Schur-Weyl duality constructed in \cite{KKK18A,KKKO15C,KKO18}, and for $\Ca_1$ and  $\Ca_N$ $(N \in \Z_{\ge1})$ of untwisted affine types $ADE$ are
 proved in \cite{HL10,HL13,Nak04} and \cite{Qin17} respectively, by approaches different from \cite{KKOP19B}.
However, by the lack of $\Z$-grading structure on $U_q'(\g)$, one can \emph{not} apply the framework in \cite{KKKO18} to those categories for monoidal categorifications directly (see also \cite[Section 4]{CW19}).

The aim of this paper is to give a criterion for a monoidal subcategory $\shc$ of $\Ca_\g$
to become a monoidal categorification of $K(\shc)$. We first introduce new invariants,
denoted also by $\Lambda(M,N)$, $\tLa(M,N)$, $\La^{\infty}(M,N)$ and $\de(M,N)$, for a pair of modules $M$ and $N$ in $\Ca_\g$, by analyzing $R$-matrices associated to $M \otimes N_z$.

We say that the universal $R$-matrix $$\Runiv_{M,N_z}
\col \cor((z))\tens_\cz(M\tens N_z)\to \cor((z))\tens_\cz(N_z\tens M)$$
 is \emph{rationally renormalizable} if there exists $c_{M,N}(z) \in \ko((z))^\times$ such that
$\Rren_{M,N_z} \seteq c_{M,N}(z)\Runiv_{M,N_z}$
sends $M\tens N_z$ to $N_z\tens M$.
In such a case, we can normalize $c_{M,N}(z) \in \ko((z))^\times$ (up to a multiple of $\ko[z^{\pm1}]^\times$)
such that $\Rren_{M,N_z}\big\vert_{z=x} \colon M \otimes N_x  \to N_x \otimes M$
does not vanish at any $x \in \ko^\times$. We call $c_{M,N}(z)$ the \emph{renormalizing coefficient} of $M$ and $N$.
We define $\tL(M,N)$ as
the order of zero of $c_{M,N}(z)$ at $z=1$.
We then define $\La(M,N)$, $\La^{\infty}(M,N)$ and $\de(M,N)$
similarly to $\tL(M,N)$ (see Definition~\ref{def: Lams} for new invariants).
Note that $\La^{\infty}(M,N) = 2\tLa(M,N)-\La(M,N)$ can be understood as a quantum affine analogue of $(\wt(V),\wt(W))$.

When $M$ and $N$ are simple modules in $\Ca_\g$, $c_{M,N}(z)$ is the ratio $d_{M,N}(z)$ to $a_{M,N}(z)$,
where $d_{M,N}(z)$ (resp.\ $a_{M,N}(z)$) denotes the \emph{denominator} (resp.\;\emph{universal coefficient}) of
the \emph{normalized $R$-matrix $\Rnorm_{M,N_z}(z)$} of $M$ and $N$, computed in~\cite{AK,DO94,KKK15B,Oh15,OT19} for fundamental representations.
Thus $\de(M,N)$ can be interpreted as the degree of zero of $d_{M,N}(z)d_{N,M}(z^{-1})$ at $z=1$ with the results in \cite{AK} (see Subsection~\ref{sec:Rmat}).

We next investigate several properties of the new invariants by using $R$-matrices and their coefficients, and prove that they
play the same role in the representation theory for quantum affine algebras as the ones for quiver Hecke algebras do.
Furthermore, new invariants provide more information arising from taking duals in $\Ca_\g$, which \emph{can not} be obtained in
the quiver Hecke algebra setting (see Remark~\ref{rem:dual}). For instances, we have
\begin{itemize}
\item $\Lambda(M,N)$ and $\Lambda^\infty(M,N)$ can be expressed in terms of $\de(M,\D^{k}N)$ for $k \in \Z$,
\item $ \Lambda^{\infty}(M,N) =\Li(N,M) = -\Lambda^{\infty}(M^*,N)= -\Lambda^{\infty}(\rd M,N)$,
\item $\Lambda^\infty(M,N) =  -\Lambda(M,D^{2n}N)=\Lambda(M,D^{-2n}N)$ for $n \gg 0$,
\item $\La(M,N)=\La(N^*,M)=\La(N,\rd M)$ and hence $\de(M,N)= \tfrac{1}{2}(\La(M,N)+\La(N,M)) = \tfrac{1}{2}(\La(M,N)+\La(M^*,N))$,
\end{itemize}
where $N^*$ (resp.\ $^*N$ and $\D^{k}N$) denotes the left (resp.\ right and $k$-th left) dual of $N$  (see Section~\ref{sec: New invs}).

With the new invariants at hand, we introduce the notions
(a) a \emph{$\Uplambda$-seed} $\Seed_\Uplambda$,
a triple $\Seed_{\Uplambda}=(\{ X_i \}_{i \in K}, L,\tB) $ consisting of a cluster $\{ X_i \}_{i \in K}$,
a skew-symmetric $K \times K$-matrix $L$ and  a $K \times \Kex$-matrix $\tB=(b_{jk})$ such that $(L\tB)_{ij}=2\delta_{ij}$, and (b)
a cluster algebra \emph{associated to} $\Seed_\Uplambda$. Here the mutation rule for a cluster algebra associated to $\Seed_\Uplambda$ is the same as the ones for quantum cluster algebras.

Finally, we introduce the  notion of a \emph{$\Uplambda$-admissible monoidal seed} in  a monoidal subcategory $\shc$ of $\Ca_\g$ by using the new invariants as follows.
 A monoidal seed $\seed=(\{M_i\}_{i\in \K},\widetilde B)$ is said to be $\Uplambda$-admissible if  it satisfies
\bna
\item
$(\Lambda^\seed\tB)_{jk}=-2\delta_{jk}$ where $\Lambda^\seed \seteq (\Lambda(M_i,M_j))_{i,j \in K}$,
\item for each $k\in\Kex$,  there exists a  real simple module $M_k'$ in $\shc$, corresponding to the mutated cluster variable $X_k'$, satisfying $\de(M_j,M_k')=\delta_{jk}$ and the short exact sequence
$$     0 \to  \sotimes_{b_{ik} >0} M_i^{\tens b_{ik}} \to M_k \tens M_k' \to  \sotimes_{b_{ik} <0} M_i^{\tens (-b_{ik})} \to 0 .$$
\ee

By employing the framework of  \cite[Section 7]{KKKO18} with new invariants and notions, we prove the main result of this paper:
\begin{maintheorem}[{Theorem~\ref{th:main}}]
For a monoidal seed $\seed=(\{M_i\}_{i \in \K},\tB)$ in a monoidal subcategory $\shc$ of $\Ca_\g$, assume the following condition{\rm:}
\begin{itemize}
\item The Grothendieck ring $\K(\shc)$ of $\shc$ is isomorphic to the cluster algebra $\mathscr{A}$ associated to
the initial seed $[\seed]\seteq \bl \{ [M_i]\}_{i \in \K}, \tB\br$,
\item $\seed$ is $\Uplambda$-admissible.
\end{itemize}
Then the category $\shc$ is a monoidal categorification of the cluster algebra $\mathscr{A}$.
\end{maintheorem}
As consequences, we can obtain the following applications (Corollary~\ref{cor:main}):
\bnum
\item For $k\in\Kex$ and the $k$-th cluster variable module $\widetilde{M}_k$ of a monoidal seed $\widetilde{\seed}$ obtained by successive mutations from the initial monoidal seed
$\seed$,  we have $\de(\widetilde{M}_k,\widetilde{M}_k')=1$.
\item  Any monoidal cluster $\{\widetilde{M}_i\}_{i\in\K}$  is a maximal real commuting family in $\shc$ (see Definition~\ref{def: max com}).
\end{enumerate}

In the forthcoming paper,
we will apply the main theorem to certain monoidal subcategories $\shc$ of $\Ca_\g$ for providing monoidal categorifications.

\smallskip

This paper is organized as follows. We give the necessary background on quantum affine algebras, their representations, and $R$-matrices, their related coefficients in Section~\ref{sec: preliminaries}.
In Section~\ref{sec: New invs} and Section~\ref{sec: Further properties}, we introduce new invariants for pairs of $\uqpg$-modules by using $R$-matrices and investigate their properties. Especially, we will show the
similarities of new invariants with the ones for quiver Hecke algebras in Section~\ref{sec: Further properties}.
In Section~\ref{sec: cluster algebra}, we briefly recall the definition of cluster algebras with the consideration on $\Uplambda$-seeds.
In Section~\ref{sec: monoidal categorification}, we prove our main result with newly introduced invariants and notions.

\section{Preliminaries} \label{sec: preliminaries}

\begin{convention}\label{convention}
\bnum
\item
For a statement $P$, $\delta(P)$ is $1$ or $0$ according that
$P$ is true or not.
\item for a field $\cor$, $a\in\cor$ and $f(z)\in\cor(z)$,
we denote
by $\zero_{z=a}f(z)$ the order of zero of $f(z)$ at $z=a$.\label{conv:zero}
\ee
\end{convention}

\subsection{Quantum affine algebras}\label{subsec:Quantum affine algebra}
Let ($\cmA$,$\wl$,$\Pi$,$\wl^\vee$,$\Pi^\vee$) be an \emph{affine Cartan datum}. It consists of   an affine Cartan matrix $\cmA=(a_{ij})_{i,j\in I}$
with a finite index set $I$, a free abelian group $\wl$ of rank $|I|+1$,
called the \emph{weight lattice}, a set $\Pi=\{\al_i \in \wl \ | \ i \in I\}$ of linearly independent elements called \emph{simple roots}, the group $\wl^\vee \seteq \Hom_\Z(\wl,\Z)$ called the \emph{coweight lattice}, and a set $\Pi^\vee=\{h_i \ | \ i \in I\} \subset \wl^\vee$ of \emph{simple coroots}.
Note that the pairing $\lan \ , \ \ran$ between $\cwl$ and $\wl$
satisfies $\langle h_i,\al_j \rangle=a_{ij}$ for all $i,j\in I$, and for each $i \in I$, there exists $\fw_i \in \wl$ such that $\lan h_j,\fw_i\ran=\delta_{ij}$ for all $j \in I$. We choose such
elements $\fw_i$ and call them the \emph{fundamental weights}.
The free abelian group $\rl \seteq \bigoplus_{i \in I}\Z\al_i\subset\wl$ is called the \emph{root lattice}. Set $\prl = \sum_{i \in I}\Z_{\ge 0}\al_i
\subset \rl$.
Similarly we set $\crl\seteq \bigoplus_{i \in I}\Z h_i\subset\wl^\vee$
and $\pcrl\seteq \sum_{i \in I}\Z_{\ge0} h_i$.

We choose \emph{the imaginary root} $\updelta= \sum_{i \in I} \mathsf{a}_i\al_i \in \prl$ and
\emph{the center} $c= \sum_{i \in I} \mathsf{c}_ih_i \in \pcrl$ such that
$\{ \la \in \wl \ | \  \lan h_i,\la \ran = 0 \text{ for every } i \in I \}=\Z\updelta$ and $\{ h \in \wl^\vee\ | \  \lan h,\al_i \ran = 0 \text{ for every } i \in I \}=\Z c$ (see \cite[Chapter 4]{Kac}).
We set $\wl_\cl  \seteq \wl / (\wl\cap\Q \updelta)\simeq\Hom(\crl,\Z)$ and call it \emph{the classical weight lattice}.
We choose $\rho \in \wl$ (resp.\ $\rho^\vee \in \wl^\vee$) such that $\lan h_i,\rho \ran=1$ (resp.\ $\lan \rho^\vee,\al_i\ran =1$) for all $i \in I$.

Set $\h \seteq \Q \tens_{\Z} \wl^\vee$. Then there exists a symmetric bilinear form $( \ , \ )$ on $\h^*$ satisfying $$\lan h_i,\la \ran= \dfrac{2(\al_i,\la)}{(\al_i,\al_i)} \qquad \text{ for any $i \in I$ and $\la \in \h^*$}.$$
We normalize the bilinear form $( \ , \ )$  by
\eq \lan c,\la \ran = (\updelta,\la)\qt{for any $\la \in \h^*$.}
\eneq

We denote by $\g$ the \emph{affine Kac-Moody algebra} associated with $(\cmA,\wl,\Pi,\wl^\vee,\Pi^\vee)$ and by $\W \seteq \lan r_i \ | \ i \in I \ran \subset GL(\h^*)$ the \emph{Weyl group} of $\g$, where
$r_i(\la) \seteq \la - \lan h_i,\la \ran \al_i$ for $\la \in \h^*$.
We will use the standard convention in~\cite{Kac}
to choose $0\in I$ except $A^{(2)}_{2n}$-case, in which case we take the longest simple root as $\al_0$.
In particular, we have always $\mathsf{a}_0=1$, while $\mathsf{c}_0=2$ or $1$
according that $\g=A^{(2)}_{2n}$ or not.

We define $\g_0$ to be the subalgebra of $\g$ generated by the Chevalley generators $e_i$, $f_i$  and $h_i$ for $i \in I_0 \seteq I \setminus \{ 0 \}$
and $\W_0$ to be the subgroup of $\W$ generated by $r_i$ for $i \in I_0$. Note that $\g_0$ is a finite-dimensional simple Lie algebra and $\W_0$ contains the longest element $w_0$.

\smallskip

Let $q$ be an indeterminate and $\ko$ be the algebraic closure of the subfield $\C(q)$
in the algebraically closed field $\corh\seteq\bigcup_{m >0}\C((q^{1/m}))$. When we deal with quantum affine algebras, we regard $\ko$ as the base field.

 For $m,n \in \Z_{\ge 0}$ and $i\in I$, we define
$q_i = q^{(\alpha_i,\alpha_i)/2}$ and
\begin{equation*}
 \begin{aligned}
 \ &[n]_i =\frac{ q^n_{i} - q^{-n}_{i} }{ q_{i} - q^{-1}_{i} },
 \ &[n]_i! = \prod^{n}_{k=1} [k]_i ,
 \ &\left[\begin{matrix}m \\ n\\ \end{matrix} \right]_i=  \frac{ [m]_i! }{[m-n]_i! [n]_i! }.
 \end{aligned}
\end{equation*}

\begin{definition} \label{Def: GKM}
The {\em quantum affine algebra} $U_q(\g)$ associated with an affine Cartan datum $(\cmA,\wl,\Pi,\wl^\vee,\Pi^\vee)$ is the associative algebra over $\ko$ with $1$ generated by $e_i,f_i$ $(i \in I)$ and
$q^{h}$ $(h \in \gamma \; \wl^{\vee})$ satisfying following relations:
\begin{enumerate}[{\rm (i)}]
\item  $q^0=1, q^{h} q^{h'}=q^{h+h'} $\quad for $ h,h' \in \gamma\; \wl^{\vee},$
\item  $q^{h}e_i q^{-h}= q^{\langle h, \alpha_i \rangle} e_i$,
$q^{h}f_i q^{-h} = q^{-\langle h, \alpha_i \rangle }f_i$\quad for $h \in \gamma^{-1}\wl^{\vee}, i \in I$,
\item  $e_if_j - f_je_i =  \delta_{ij} \dfrac{K_i -K^{-1}_i}{q_i- q^{-1}_i }, \ \ \text{ where } K_i=q_i^{ h_i},$
\item  $\displaystyle \sum^{1-a_{ij}}_{k=0}
(-1)^ke^{(1-a_{ij}-k)}_i e_j e^{(k)}_i =  \sum^{1-a_{ij}}_{k=0} (-1)^k
f^{(1-a_{ij}-k)}_i f_jf^{(k)}_i=0 \quad \text{ for }  i \ne j, $
\end{enumerate}
where $e_i^{(k)}=e_i^k/[k]_i!$ and $f_i^{(k)}=f_i^k/[k]_i!$.
\end{definition}

Let us denote by $U_q^+(\g)$ (resp.\ $U_q^-(\g)$) the subalgebra of $U_q(\g)$ generated by $e_i$ (resp.\ $f_i$) for $i \in I$.
We denote by $U_q'(\g)$ the subalgebra of $U_q(\g)$ generated by $e_i,f_i,K^{\pm 1}_i$ $(i \in I)$ and we call it \emph{also} the quantum
affine algebra. Throughout this paper, we mainly deal with $U_q'(\g)$.

We use the coproduct $\Delta$ of $\uqpg$ given by
\begin{align}\label{eq: comultiplication}
\Delta(q^h)=q^h \tens q^h, \ \ \Delta(e_i)=e_i \tens K_i^{-1}+1 \tens e_i, \ \Delta(f_i)=f_i \tens 1 +K_i \tens f_i.
\end{align}

Let us denote by $ \ \bar{ } \ $ the bar involution of $U_q'(\g)$ defined as follows:
\[ q^{1/m} \to q^{-1/m}, \qquad\qquad
e_i \mapsto e_i, \qquad\qquad f_i \mapsto f_i, \qquad\qquad K_i \mapsto K_i^{-1}.
\]

We denote by $\uqm$ the category of finite-dimensional integrable
$\uqpg$-modules;
i.e., finite-dimensional modules $M$ with a weight decomposition
$$M=\soplus_{\la\in\wl_\cl}M_\la  \qquad \text{ where } M_\la=\st{u\in M\mid K_iu=q_i^{\ang{h_i,\la}}}.$$
Note that $\uqm$ is a monoidal
category with the coproduct in~\eqref{eq: comultiplication}.
It is known that the Grothendieck ring $K(\uqm)$ is a commutative ring.
A simple module $M$ in $\uqm$ contains a non-zero vector $u$ of weight $\lambda\in \wl_\cl$ such that (i) $\langle h_i,\lambda \rangle \ge 0$ for all $i \in I_0$,
(ii) all the weight of $M$ are contained in $\lambda - \sum_{i \in I_0} \Z_{\ge 0} \cl(\alpha_i)$, where $\cl\colon \wl\to \wl_\cl$ denotes the canonical projection.
Such a $\la$ is unique and $u$ is unique up to a constant multiple. We call $\lambda$ the {\it dominant extremal weight} of $M$ and $u$ a {\it dominant extremal weight vector} of $M$.

For an integrable $U_q'(\g)$-module $M$, the {\it affinization} $M_z$ of $M$ is the $U_q(\g)$-module
$$M_z = \soplus_{\lambda \in \wl} (M_z)_\lambda \quad \text{ with } \quad  (M_z)_\lambda=M_{\cl(\lambda)}.$$
Here the actions $e_i$  and $f_i$ are defined in a way that they commute with the canonical projection $\cl: M_z \to M$.

We denote by $z_M \colon M_z \to M_z$ the $U_q'(\g)$-module automorphism of weight $\updelta$ defined by
$(M_z)_{\lambda} \overset{\sim}{\longrightarrow} (M_z)_{\lambda+\updelta}$. For $x \in \ko^\times$, we define
$$M_x \seteq M_z / (z_M -x)M_z.$$
We call $x$ a {\it spectral parameter}. Note that, for a module $M$ in $\uqm$ and $x \in \ko^\times$, $M_x$ is also contained in $\uqm$.
The functor $T_x$ defined by $T_x(M)=M_x$ is an endofunctor of $\uqm$
which commutes with tensor products.

Let us take a  section  $\iota\col \wl_\cl\monoto \wl$ of  $\cl\col\wl\to\wl_\cl$
such that $\iota\cl (\al_i)=\al_i$ for all $i\in I_0$.
 For $u\in M_\lambda$ ($\lambda \in P_\cl$) and an indeterminate $z$, let
us denote by $u_z\in (M_z)_{\iota(\lambda)}$ the element such that
$\cl(u_z)=u$. With this notation,  we have
$$e_i(u_z)=z^{\delta_{i,0}}(e_iu)_z, \quad
f_i(u_z)=z^{-\delta_{i,0}}(f_iu)_z, \quad K_i(u_z)=(K_iu)_z.
$$
Then we have $M_z \simeq \cz\tens M$, and
the automorphism $z_M$ on $M_z$ corresponds to
the multiplication of $z$ on $\cz\tens M$.
Thus $u_z$ is the element $1 \tens u \in \cz \tens M$ for $u \in M$. We also use $M_{z_{M}}$ instead of $M_z$ to emphasize $z$
as the automorphism on $M_z$ of weight $\updelta$.

For each $i \in I_0$, we set $$\varpi_i \seteq {\rm
gcd}(\mathsf{c}_0,\mathsf{c}_i)^{-1}\cl(\mathsf{c}_0\Lambda_i-\mathsf{c}_i \Lambda_0) \in \wl_\cl.$$
Then $\wl_\cl^0\seteq\{\la\in\wl_\cl\mid \ang{c,\la}=0\}$ is equal to
$\soplus_{i\in I_0}\Z\varpi_i$. Moreover, for any $i\in I_0$, there
exists a unique simple module $V(\varpi_i)$ in $\uqm$ satisfying certain conditions (see \cite[\S 5.2]{Kas02}), which is called the \emph{fundamental module of weight $\varpi_i$}.
The dominant extremal weight of $V(\varpi_i)$ is $\varpi_i$.

For a $U_q'(\g)$-module $M$, we denote by $\overline{M}=\{ \bar{u} \mid u \in M \}$
the $U_q'(\g)$-module defined by $x \bar{u} \seteq \overline{\ms{2mu}\overline{x} u\ms{2mu}}$ for $x \in U_q'(\g)$.
Then we have
\begin{align} \label{eq: bar structure}
\overline{M_a} \simeq (\overline{M})_{\,\overline{a}}, \qquad\qquad \overline{M \otimes N} \simeq \overline{N} \otimes \overline{M}.
\end{align}
Note that $V(\varpi_i)$ is \emph{bar-invariant}; i.e., $\overline{V(\varpi_i)}\simeq V(\varpi_i)$ (see \cite[Appendix A]{AK}).

\begin{remark}[{\cite[\S 1.3]{AK}}] \label{rem:m_i}
Let $m_i$ be a positive integer such that
$$\W\pi_i\cap\bigl(\pi_i+\Z\updelta\bigr)=\pi_i+\Z m_i\updelta,$$
where $\pi_i$ is an element of $\wl$ such that $\cl(\pi_i)=\varpi_i$.
We have $m_i=(\al_i,\al_i)/2$ in the case when $\g$ is the dual of an untwisted affine algebra, and $m_i=1$ otherwise.
Then, for $x,y\in \ko^\times$, we have
$$V(\varpi_i)_x \simeq V(\varpi_i)_y  \quad \text{ if and only if }   \quad   x^{m_i}=y^{m_i}.$$
\end{remark}

For a module $M$ in $\uqm$, let us denote the right and the left dual of $M$ by $\rd M$ and $M^*$, respectively.
That is, we have isomorphisms
\begin{align*}
&\Hom_{U_q'(\g)}(M\hspace{-.4ex} \tens \hspace{-.4ex} X,Y) \hspace{-.2ex} \simeq \hspace{-.2ex} \Hom_{U_q'(\g)}(X, \hspace{-.4ex} \rd M \hspace{-.4ex} \tens \hspace{-.4ex} Y),\
\Hom_{U_q'(\g)}(X \hspace{-.4ex} \tens \hspace{-.4ex} \rd M,Y)\hspace{-.2ex} \simeq \hspace{-.2ex} \Hom_{U_q'(\g)}(X, Y \hspace{-.4ex} \tens \hspace{-.4ex} M),\\
&\Hom_{U_q'(\g)}(M^* \hspace{-.4ex} \tens \hspace{-.4ex}  X,Y)\hspace{-.2ex} \simeq \hspace{-.2ex} \Hom_{U_q'(\g)}(X, M \hspace{-.4ex}  \tens \hspace{-.4ex}  Y),\
\Hom_{U_q'(\g)}(X \hspace{-.4ex} \tens \hspace{-.4ex}  M,Y)\hspace{-.2ex} \simeq \hspace{-.2ex} \Hom_{U_q'(\g)}(X, Y \hspace{-.4ex}  \tens \hspace{-.4ex} M^*),
\end{align*}
which are functorial in $U_q'(\g)$-modules $X$ and $Y$.

Hence, we have the evaluation morphisms
$$M\tens \rd M\to\one,\quad M^*\tens M\to\one$$
and the co-evaluation morphisms
$$\one\to \rd M\tens M,\quad \one\to M\tens M^*.$$

Note the followings (see \cite[Appendix A]{AK}):
\begin{enumerate}[{\rm (i)}]
\item For any module $M$ in $\uqm$, we have
$$  M^{**} \simeq M_{q^{-2(\updelta,\rho)}} \quad \text{ and }  {}^{**}\ms{-1mu}M \simeq M_{q^{2(\updelta,\rho)}}.$$
\item The duals of $V(\varpi_i)_x$ $(x\in \cor^\times)$ satisfy:
\begin{equation}\label{eq: dual}
\bigl( V(\varpi_i)_x\bigr)^*  \simeq   V(\varpi_{i^*})_{(p^*)^{-1}x},
\qquad {}^*\bigl(V(\varpi_i)_x \bigr) \simeq   V(\varpi_{i^*})_{p^*x}
\end{equation}
where $p^* \seteq (-1)^{\langle \rho^\vee ,\updelta \rangle}q^{\ang{c,\rho}}$ and $i^*\in I_0$
is defined by  $\al_{i^*}=-w_0\,\al_i$.
\end{enumerate}

We say that a $U_q'(\g)$-module $M$ is {\em good}
if it has a {\em bar involution}, a crystal basis with {\em simple
crystal graph}, and a {\em global basis} (see~\cite{Kas02} for
the precise definition). It is known that the fundamental representations are good modules.

\begin{definition}
We say that  a $U_q'(\g)$ module $M$ is {\em quasi-good} if
$$ M \simeq V_c $$
for some good module $V$ and $c \in \cor^\times$.
\end{definition}

Note that every quasi-good module is a simple $U_q'(\g)$-module. Moreover the tensor product $M^{\otimes k} \seteq \underbrace{M \tens \cdots \tens M}_{k\text{-times}}$ for a quasi-good module $M$ and $k \in \Z_{\ge 1}$
is again quasi-good.

For simple modules $M$ and $N$ in $\uqm$,
we say that $M$ and $N$ {\em commute}
or $M$ commutes with $N$ if $M\tens N\simeq N\tens M$.
We say that $M$ and $N$ \emph{strongly commute} or $M$ \emph{strongly commutes with} $N$ if $M\tens N$ is simple.
When simple modules $M$ and $N$ strongly commute, they commute.
Note that $M\tens N$ is simple if and only if $N\tens M$ is simple, since $K(\uqm)$ is a commutative ring.

Also, when the simple modules $M_t$ $(1 \le t \le m)$ strongly commute with
 each other, it is proved in \cite{Her10S} that
$$  M_{1} \otimes \cdots \otimes M_m \simeq M_{\sigma(1)} \otimes  \cdots \otimes M_{\sigma(m)} \text{ is simple}     $$
for every element $\sigma$ in the symmetric group $\sym_m$ on $m$-letters.
We say that a simple module $L$ in $\Ca_\g$  is \emph{real} if $L$ strongly commutes with itself, i.e., if $L\tens L$ is simple.
Note that quasi-good modules are real.

\subsection{R-matrices, universal and renormalizing coefficients} \label{sec:Rmat}

In this subsection, we review the notion of $R$-matrices on $\uqpg$-modules and their coefficients
by following mainly \cite[\S 8]{Kas02} and \cite[Appendices A and B]{AK}. Let us choose the \emph{universal $R$-matrix}
in the following way: take a basis $\{P_\nu\}_\nu$ of $U_q^+(\g)$
and a basis $\{Q_\nu\}_\nu$ of $U_q^-(\g)$ dual to each other with respect to a suitable coupling between
$U_q^+(\g)$ and $U_q^-(\g)$. Then for $\uqpg$-modules $M$ and $N$ define
\begin{align}
\Runiv_{M,N}(u\otimes v)=q^{(\wt(u),\wt(v))} \sum_\nu P_\nu v\otimes Q_\nu u\,,
\end{align}
so that
$\Runiv_{M,N}$ gives a $\uqpg$-linear homomorphism from $M\otimes N$ to $N\otimes M$ provided that the infinite sum has a meaning.

For modules $M$ and $N$ in $\uqm$, it is known that $\Runiv_{M,N_z}$ converges in the $z$-adic topology. Hence, it induces a morphism of
$\ko((z))\tens\uqpg$-modules
\begin{align}\label{eq: renom1}
\Runiv_{M,N_z} \colon \ko((z))\tens_{\ko[z^{\pm1}]} (M \tens N_z) \To \ko((z))\tens_{\ko[z^{\pm1}]} (N_z\tens M).
\end{align}
Moreover,  $\Runiv_{M,N_z}$ is an isomorphism.

It is known that $\Runiv$ satisfies the following properties:
the following diagram commutes
\eq\label{eq: commutativity}
&&
\scalebox{.89}{\xymatrix@C=1ex{
\ko((z))\tens_{\ko[z^{\pm1}]} (M\tens N\tens L_z)
 \ar[rd]_(.4){M\tens \Runiv_{N,L_z}}\ar[rr]^
{\Runiv_{M\tens N,L_z}}
&&
\ko((z))\tens_{\ko[z^{\pm1}]} (L_z\tens M\tens N)\\
&\ko((z))\tens_{\ko[z^{\pm1}]} (M\tens L_z\tens N)
\ar[ru]_(.6){\Runiv_{M,L_z}\tens N}
}}
\eneq
for $L$, $M$, $N$ in $\uqm$.

\medskip
Let $M$ and $N$ be non-zero modules in $\uqm$.
If there exists $a(z) \in \ko((z))$
such that
$$a(z) \Runiv_{M,N_z}\bl M\tens N_z\br\subset N_z\tens M, $$
then we say that $\Runiv_{M,N_z}$ is \emph{\rr}.
In this case, we can choose
$c_{M,N}(z) \in \ko((z))^\times$ as $a(z)$ such that,
 for any $x \in \ko^\times$, the specialization
of $\Rren_{M,N_z} \seteq c_{M,N}(z)\Runiv_{M,N_z}
\col M \otimes N_z \to N_z \otimes M$ at $z=x$
$$  \Rren_{M,N_z}\big\vert_{z=x} \colon M \otimes N_x  \to N_x \otimes M$$
does not vanish.  Such $\Rren_{M,N_z}$ and $c_{M,N}(z)$ are unique up to a multiple of $\cz^\times = \bigsqcup_{\,n \in \Z}\ko^\times z^n$. We call $c_{M,N}(z)$  the
\emph{\rc}.

We write
$$ \rmat{M,N} \seteq \Rren_{M,N_z}\vert_{z=1} \colon M \otimes N \to N \otimes M,$$
and call it \emph{$R$-matrix}. The $R$-matrix $\rmat{M,N}$ is well-defined up to a constant multiple when $\Runiv_{M,N_z}$ is \rr. By the definition, $\rmat{M,N}$ never vanishes.

\medskip
Now assume that $M$ and $N$ are
simple $\uqpg$-modules in $\uqm$.
Then $\ko(z)\otimes_{\ko[z^{\pm1}]}  \big(M \otimes N_z \big)$ is a simple $\ko(z)\tens\uqpg$-module (\cite[Proposition 9.5]{Kas02}).

Furthermore, we have the following.
Let $u$ and $v$
be dominant extremal weight vectors of $M$ and $N$,
respectively.
Then there exists $a_{M,N}(z)\in\ko[[z]]^\times$
such that
$$\Runiv_{M,N_z}\big( u \tens v_z\big)=
a_{M,N}(z)\big( v_z\tens u \big).$$
Then $\Rnorm_{M,N_z}\seteq a_{M,N}(z)^{-1}\Runiv_{M,N_z}\vert_{\;\ko(z)\otimes_{\ko[z^{\pm1}]} ( M \otimes N_z) }$
induces a unique $\ko(z)\tens\uqpg$-module isomorphism
\begin{equation*}
\Rnorm_{M, N_z} \colon
\ko(z)\otimes_{\ko[z^{\pm1}]} \big( M \otimes N_z\big) \longisoto\ko(z)\otimes_{\ko[z^{\pm1}]}  \big( N_z \otimes M \big)
\end{equation*}
satisfying
\begin{equation*}
\Rnorm_{M, N_z}\big( u  \otimes v_z\big) = v_z\otimes u .
\end{equation*}
Hence, the universal $R$-matrix
$\Runiv_{M,N_z}$ is \rr.
We call $a_{M,N}(z)$ the {\it universal coefficient} of $M$ and $N$, and $\Rnorm_{M,N_z}$ the {\em normalized $R$-matrix}.

Similarly there exists a unique $\ko(z)\tens\uqpg$-module isomorphism
\eqn
&&
\Rnorm_{M_z, N} \colon
\ko(z)\otimes_{\ko[z^{\pm1}]} \big( M_z \otimes N\big) \longisoto \ko(z)\otimes_{\ko[z^{\pm1}]}  \big( N\otimes M_z \big)
\eneqn
satisfying
\begin{equation*}
\Rnorm_{M, N}\big( u_z  \otimes v\big) = v\otimes u_z .
\end{equation*}
Note that
$\Rnorm_{M_z,N}=T_z\circ\Rnorm_{M,N_w}$
with $w=1/z$.
Here, for $x\in\cor(z)$, the functor $T_x$ is the endofunctor of the category of
$\cor(z)\tens\uqpg$-modules $L$ given by $T_x(L)=L_x$

Let $d_{M,N}(z) \in \ko[z]$ be a monic polynomial of the smallest degree such that the image of $d_{M,N}(z)
\Rnorm_{M,N_z}(M\tens N_z)$ is contained in $N_z \otimes M$. We call $d_{M,N}(z)$ the {\em denominator of $\Rnorm_{M,N_z}$}. Then we have
\begin{equation}
\Rren_{M,N_z}= d_{M,N}(z)\Rnorm_{M,N_z}
\col M \otimes N_z \To N_z \otimes M
\qt{up to a multiple of $\cz^\times$.}
\end{equation}
Hence, we have
\begin{align} \Rren_{M,N_z} =a_{M,N}(z)^{-1}d_{M,N}(z)\Runiv_{M,N_z}
\quad \text{and} \quad  c_{M,N}(z)= \dfrac{d_{M,N}(z)}{a_{M,N}(z)}
\end{align}
up to a multiple of $\ko[z^{\pm1}]^\times$.

Since $\ko(z)\otimes_{\ko[z^{\pm1}]}  \big(M \otimes N_z \big)$ is a simple $\ko(z)\tens\uqpg$-module (\cite[Proposition 9.5]{Kas02}),
we have
\eq
&&\Hom_{\cz\tens\uqpg}(M\tens N_z,N_z\tens M)=\cz\Rren_{M,N_z}.
\eneq
Similarly there exists a $\cz\tens\uqpg$-linear homomorphism
$\Rren_{M_z,N}\col M_z\tens N\to N\tens M_z$ such that
\eq
&&\Hom_{\cz\tens\uqpg}(M_z\tens N,N\tens M_z)=\cz\Rren_{M_z,N}.
\eneq
The homomorphism
$\Rren_{M_z,N}$ is unique up to a multiple of $\czt$.
We have
\eq
\Rren_{M_z,N}=d_{M,N}(z^{-1})\;\Rnorm_{M_z,N}\quad\mod\czt.
\eneq
In particular, we have
\eq
\Rren_{N_z,M}\circ\Rren_{M,N_z}=d_{M,N}(z)d_{N,M}(z^{-1})\id_{M\tens N_z}
\quad\mod \ \czt.\label{eq:RR}
\eneq

\begin{theorem}[{ \cite{AK,Chari,Kas02,KKKO15}; see also \cite[Theorem 2.2]{KKK18A}}]  \label{Thm: basic properties}\hfill
\bnum
\item For good modules $M$ and $N$, the zeroes of $d_{M,N}(z)$ belong to
$\C[[q^{1/m}]]q^{1/m}$ for some $m\in\Z_{>0}$.
\item \label{it:comm} For simple modules $M$ and $N$ such that one of them is real, $M_x$ and $N_y$ strongly commute to each other if and only if $d_{M,N}(z)d_{N,M}(1/z)$ does not vanish at $z=y/x$.
\item  Let $M_k$ be a good module
with a dominant extremal vector $u_k$ of weight $\lambda_k$, and
$a_k\in\ko^\times$ for $k=1,\ldots, t$.
Assume that $a_j/a_i$ is not a zero of $d_{M_i, M_j}(z) $ for any
$1\le i<j\le t$. Then the following statements hold.
\bna
\item  $(M_1)_{a_1}\otimes\cdots\otimes (M_t)_{a_t}$ is generated by $u_1\otimes\cdots \otimes u_t$.
\item The head of $(M_1)_{a_1}\otimes\cdots\otimes (M_t)_{a_t}$ is simple.
\item Any non-zero submodule of $(M_t)_{a_t}\otimes\cdots\otimes (M_1)_{a_1}$ contains the vector $u_t\otimes\cdots\otimes u_1$.
\item The socle of $(M_t)_{a_t}\otimes\cdots\otimes (M_1)_{a_1}$ is simple.
\item  Let $\rmat{}: (M_1)_{a_1}\otimes\cdots\otimes (M_t)_{a_t} \to (M_t)_{a_t}\otimes\cdots\otimes (M_1)_{a_1}$  be the specialization of $R^{{\rm norm}}_{M_1,\ldots, M_t}
\seteq\prod\limits_{1\le j<k\le t}R^{{\rm norm}}_{M_j,\,M_k}$ at $z_k=a_k$.
Then the image of $\rmat{}$ is simple and it coincides with the head of $(M_1)_{a_1}\otimes\cdots\otimes (M_t)_{a_t}$
and also with the socle of $(M_t)_{a_t}\otimes\cdots\otimes (M_1)_{a_1}$.
\end{enumerate}
\item\label{item5}
For a simple integrable $U_q'(\g)$-module $M$, there exists
a finite sequence
\begin{align} \label{eq: ordered pair}
\big((i_1,a_1),\ldots, (i_t,a_t)\big) \in  I_0\times \mathbf{k}^\times
\end{align}
which satisfies the following condition$:$
for any $\sigma\in\sym_t$ such that
$$d_{V(\varpi_{i_{\sigma(k)}}),V(\varpi_{i_{\sigma(k')}})}(a_{\sigma(k')}/a_{\sigma(k)}) \not=0\quad\text{for $1\le k<k'\le t$,}$$
$M$ is isomorphic to the head of $V(\varpi_{i_{\sigma(1)}})_{a_{\sigma(1)}}\otimes\cdots\otimes V(\varpi_{i_{\sigma(t)}})_{a_{\sigma(t)}}$.

Moreover, such a sequence $\big((i_1,a_1),\ldots, (i_t,a_t)\big)$
is unique up to permutation. In particular, $M$ has the dominant extremal weight $\sum_{k=1}^t \varpi_{i_k}$.
\end{enumerate}
\end{theorem}

{}From the above theorem, for each simple module $M$ in $\uqm$, we can associate a  multiset of pairs  $\{ (i_k,a_k) \in I_0 \times \ko^\times \}_{1 \le  k \le t} $ satisfying
the conditions in Theorem~\ref{Thm: basic properties}~\eqref{item5}.
We call $\{ (i_k,a_k) \in I_0 \times \ko^\times \}_{1 \le k \le t} $ of $M$ \emph{the multipair associated to $M$}, and
write $$M=S((i_1,a_1),\ldots,(i_t,a_t)).$$

\Prop\label{prop:cstar}
Let $M$ and $N$ be non-zero modules in $\uqm$, and $a\in\cor^\times$ such that $\Runiv_{M,N_z}$ is rationally renormalizable. 
Then we have
\eqn
&& c_{M,N}(z)=c_{M^*,N^*}(z)=c_{\rd M,\rd N}(z),\\
&& c_{M_a,N}(z)= c_{M,N}(a^{-1}z),\quad c_{M,N_a}(z)= c_{M,N}(az).
\eneqn
\enprop
\Proof
The first assertion follows from
$\bl\Runiv_{M,N_z}\br^*=\Runiv_{M^*,N^*_z}$: that is,
$$\xymatrix@C=13ex{
(N_z\tens M)^*\ar@{-}[d]^\bwr\ar[r]_{\bl\Runiv_{M,N_z}\br^*}&(M\tens N_z)^*
\ar@{-}[d]^\bwr
\\
M^*\tens N_z^*\ar[r]_{\Runiv_{M^*,N_z^*}}
&(N_z)^*\tens M^*
}$$
commutes (\cite{FR92}). The second follows from the first and
the others are  trivial.
\QED

\begin{proposition}[{\cite[(A14), (A15), Proposition A.1, Lemma C.15]{AK}}]  \label{prop: aMN}
Let $M$ and $N$ be simple modules in $\uqm$.
\bnum
\item We have
\eq&&\hs{7ex}\ba{lll}
&&a_{M,N}(z)=a_{M^*,\,N^*}(z)=a_{\rd M,\, \rd N}(z),\\
&&d_{M,N}(z)=d_{M^*,\,N^*}(z)=d_{\rd M,\, \rd N}(z),\\[.5ex]
&&a_{M,N}(z)=a_{M_x,\,N_x}(z),\;
d_{M,N}(z)=d_{M_x,\,N_x}(z)\ \text{for any $x\in\cor^\times$.}
\ea
\label{eq:ainv}\eneq
\item
$
a_{M,N}(z)a_{\rd M,N}(z) \equiv \dfrac{d_{M,N}(z)}{d_{N,\rd M}(z^{-1})}
\quad \mod \ \ko[z^{\pm1}]^\times$.
\label{MN2}
\end{enumerate}
\end{proposition}

We set
\begin{align} \label{eq: varphi and bracket}
 \varphi(z) \seteq \prod_{s=0}^\infty (1-\tp^{s}z)\in\ko[[z]]\subset\corh[[z]].
\end{align}
Here $\tp\seteq p^{*\,2}=q^{2\ang{c,\rho}}=q^{2\sum_{i\in I}\mathsf{c}_i}$.
We have
$$\vphi(z)=\sum_{m=0}^\infty(-1)^m\;\dfrac{\tp^{m(m-1)/2}}{\prod_{k=1}^m(1-\tp^k)}\;z^m.$$

For $i,j \in I_0$, set
\eq &&\ba{rl}
a_{i,j}(z) &\seteq a_{V(\varpi_i),V(\varpi_j)}(z),\\
d_{i,j}(z) &\seteq d_{V(\varpi_i),V(\varpi_j)}(z).\ea
\eneq
Then the universal coefficient $a_{i,j}(z)$ is obtained as follows (see \cite[Appendix A]{AK}):
\begin{align}\label{eq: aij}
a_{i,j}(z) \equiv \dfrac{\prod_{\mu} \varphi(p^*y_\mu z) \varphi(p^*\overline{y_\mu}z) }{\prod_{\nu} \varphi(x_\nu z) \varphi(p^{*2}\overline{x_\nu}z)} \qquad {\rm mod} \ \ko[z^{\pm1}]^\times,
\end{align}
where
$$d_{i,j}(z)=\prod_{\nu}(z-x_\nu) \quad \text{ and } \quad d_{i^*,\,j}(z)=\prod_{\mu}(z-y_\mu).$$

\begin{example}
For the fundamental representations $V(\varpi_i)$'s over $U_q'(A_{n-1}^{(1)})$
($i \in I_0=\st{1,\ldots,n-1}$),
the denominators $d_{i,j}(z) \seteq d_{V(\varpi_i),V(\varpi_j)}(z)$
and the universal coefficients of $ a_{i,j}(z)\seteq a_{V(\varpi_i),V(\varpi_j)}(z)$ are given as follows:
\begin{align*}
d_{i,j}(z)= \hspace{-3ex}\displaystyle\prod_{s=1}^{ \min(i,j,n-i,n-j)} \hspace{-3ex} \big(z-(-q)^{2s+|i-j|}\big) \quad \text{and} \quad   a_{i,j}(z)\equiv \dfrac{[\,|i-j|\,]\,[2n-|i-j|\,]}{[i+j]\,[2n-i-j]} \quad {\rm mod} \ \ko[z^{\pm1}]^\times
\end{align*}
where $[a] \seteq \varphi((-q)^{a}z)$. Note that $p^* = (-q)^{n}$ in this case.
\end{example}

\begin{remark}
The denominators of the normalized $R$-matrices $d_{i,j}(z)$ and hence the universal coefficients $a_{i,j}(z)$ were calculated in~\cite{AK,DO94,KKK15B,Oh15} for the classical affine types
and in \cite{OT19} for the exceptional affine types (see also~\cite{FR15,Fu18,KMN2,KMOY07,Ya98}).
\end{remark}

\begin{lemma}[{\cite[Lemma 3.10]{KKKO15}}]\label{lem: middle} Let $M_k$ be a module in $\uqm$ $(k=1,2,3)$.
Let $X$ be a $\uqpg$-submodule of $M_1\tens M_2$ and
$Y$ a $\uqpg$-submodule of $M_2\tens M_3$
such that $X\tens M_3\subset M_1\tens Y$ as submodules of
$M_1\tens M_2\tens M_3$.
Then there exists a $\uqpg$-submodule $N$ of $M_2$
such that
$X\subset M_1\tens N$ and $N\tens M_3\subset Y$.
\end{lemma}

\Prop[{\cite[Corollary 3.11]{KKKO15}}] \label{prop:rcomp}\hfill
\bnum
\item Let $M_k$ be a module in $\uqm$ $(k=1,2,3)$, and let $\varphi_1\col L\to M_2\tens M_3$ and $\varphi_2\col M_1\tens M_2\to L'$
be {\em non-zero} morphisms. Assume further that $M_2$ is a simple module.
Then the composition
\begin{align*}
M_1\tens L \To[M_1\tens \varphi_1] M_1\tens M_2\tens M_3\To[\varphi_2\tens M_3]
L'\tens M_3
\end{align*}
does not vanish.
\item Let $M$, $N_1$ and $N_2$ be non-zero modules in $\uqm$, and assume that
$\Runiv_{N_k,M_z}$ is \rr for $k=1,2$.
Then
$\Runiv_{N_1\tens N_2,M_z}$ is \rr, and
we have
$$\dfrac{c_{N_1,M}(z)c_{N_2,M}(z)}{c_{N_1\tens N_2, M}(z)}\in \cz.$$
If we assume further that $M$ is simple, then we have
$$   c_{N_1 \otimes N_2,M}(z)\equiv c_{N_2,M}(z) c_{N_1,M}(z) \mod \cz^\times $$
and the following diagram commutes up to a constant multiple{\rm:}
\eq
\xymatrix@C=12ex
{ N_1\tens N_2 \tens M \ar[r]_{N_1 \tens \rmat{N_2,M}}\ar@/^2pc/[rr]|-{\rmat{N_1\tens N_2, \ M}}
&N_1\tens M\tens N_2\ar[r]_{\rmat{N_1,M}\tens\, N_2}& M \tens N_1\tens N_2.
}\label{eq:tensr}
\eneq

\item Let $M$, $N_1$ and $N_2$ be  non-zero modules in $\uqm$, and assume that
$\Runiv_{M,\, (N_k)_z}$ is \rr for $k=1,2$.
Then
$\Runiv_{M,\,(N_1\tens N_2)_z}$ is \rr, and
we have
$$\dfrac{c_{M,N_1}(z)c_{M,N_2}(z)}{c_{M,N_1\tens N_2}(z)}\in \cz.$$
If we assume further that $M$ is simple, then
we have
$$   c_{N_1 \otimes N_2,M}(z)\equiv c_{N_2,M}(z) c_{N_1,M}(z) \mod \cz^\times $$
and
the following diagram commutes up to a constant multiple{\rm:}
\eq
\xymatrix@C=12ex
{M\tens N_1\tens N_2  \ar[r]_{\rmat{M,N_1}\tens N_2}
\ar@/^2pc/[rr]|-{\rmat{M,\;N_1\tens N_2}}
&N_1\tens M\tens N_2\ar[r]_{N_1\tens\rmat{M,N_2}\tens\, N_2}&
\tens N_1\tens N_2\tens M.
}
\eneq
\ee
\enprop

\begin{proof}
(i) and the commutativity of \eqref{eq:tensr} are
 nothing but \cite[Corollary 3.11]{KKKO15}.
Since
$$\scalebox{.8}{
\xymatrix@C=5ex
{\cor((z))\tens_{\cz}(N_1\tens N_2 \tens M_z)
\ar[rd]_{N_1 \tens \Runiv_{N_2,M_z}}\ar[rr]^{\Runiv_{N_1\tens N_2, \ M}}
&& \cor((z))\tens_{\cz}( M_z \tens N_1\tens N_2)\\
&\cor((z))\tens_{\cz}(N_1\tens M_z\tens N_2)
\ar[ru]_{\Runiv_{N_1,M_z}\tens\, N_2}
}
}$$
commutes,
the diagram
$$
\xymatrix@C=20ex
{N_1\tens N_2 \tens M_z\ar[r]_{N_1 \tens c_{N_2,M}(z)\Runiv_{N_2,M_z}}
\ar@/^2pc/[rr]|-{c_{N_2,M}(z) c_{N_1,M}(z)\Runiv_{N_1\tens N_2, \ M_z}}
&N_1\tens M_z\tens N_2\ar[r]_{ c_{N_1,M}(z)\Runiv_{N_1,M_z}\tens\, N_2}&
M_z \tens N_1\tens N_2}
$$
commutes.
Hence $\Runiv_{N_1\tens N_2,M_z}$ is \rr, and
we have ${c_{N_1,M}(z)c_{N_2,M}(z)}\in \cz{c_{N_1\tens N_2, M}(z)}$.

If $M$ is simple, then (i) implies that
$c_{N_2,M}(z) c_{N_1,M}(z)\Runiv_{N_1\tens N_2, \ M_z}$ never vanishes at any $z=a\in\cor^\times$.
Hence $c_{N_2,M}(z) c_{N_1,M}(z)\Runiv_{N_1\tens N_2, \ M_z}\equiv
\Rren_{N_1\tens N_2, \ M_z} \mod \cz^\times$, which implies
$c_{N_1 \otimes N_2,M}(z)\equiv c_{N_2,M}(z) c_{N_1,M}(z) \mod \cz^\times$.

\smallskip
The proof of (iii) is similar.
\end{proof}

\Prop\label{prop:subqc}
Let $M$ and $N$ be modules in $\uqm$, and let
$M'$ and $N'$ be a non-zero subquotient of $M$ and $N$, respectively.
Assume that $\Runiv_{M,N_z}$ is \rr.
Then $\Runiv_{M',N'_z}$ is \rr, and
$c_{M,N}(z)/c_{M'.N'}(z)\in\cz$.
\enprop
\Proof
We shall show that
$\Runiv_{M',N_z}$ is \rr and
$c_{M,N}(z)/c_{M'.N}(z)\in\cz$ for a non-zero quotient $M'$ of $M$.
We have
a commutative diagram
\eqn
\xymatrix@R=5ex@C=13ex{
\cor((z))\tens_{\cz}(M\otimes N_z) \ar@{->>}[d]
\ar[r]^{c_{M,N}(z)\Runiv_{M,N_z}} &
\cor((z))\tens_{\cz}(N_z\otimes M)\ar@{->>}[d]\\
\cor((z))\tens_{\cz}(M'\otimes N_z)
\ar[r]^{c_{M,N}(z)\Runiv_{M',N_z}} &
\cor((z))\tens_{\cz}(N_z\otimes M')
}
\eneqn
which induces
\eqn
\xymatrix@C=14ex{
M\otimes N_z \ar@{->>}[d]
\ar[r]^{\Rren_{M,N_z}} &
N_z\otimes M\ar@{->>}[d]\\
M'\otimes N_z
\ar[r]^{c_{M,N}(z)\Runiv_{M',N_z}} &
N_z\otimes M'\;.
}
\eneqn
Hence $\Runiv_{M',N_z}$ is \rr and
$c_{M,N}(z)\in c_{M'.N}(z)\cz$.

Similarly $\Runiv_{M',N_z}$ is \rr and
$c_{M,N}(z)/c_{M'.N}(z)\in\cz$ for any non-zero submodule of $M'$ of $M$,
and hence for any  non-zero subquotient of $M'$ of $M$.

We can argue similarly for non-zero subquotients $N'$ of $N$.
\QED

\begin{theorem}[{\cite{KKKO15}}]\label{thm: KKKo15 main} Let $M$ and $N$ be simple modules in $\uqm$ and assume that one of them is real. Then
\bnum
\item $\Hom(M\tens N,N\tens M)=\cor\,\rmat{M,N}$.
\label{runiq}
\item $M \otimes N$ and $N \otimes M$ have simple socles and simple heads.
\item Moreover, ${\rm Im}(\rmat{M,N})$ is isomorphic to the head of $M \otimes N$ and the socle of $N \otimes M$.
\item  $M \otimes N$ is simple whenever its head and its socle
are isomorphic to each other.
\end{enumerate}
\end{theorem}
Note that \eqref{runiq} is not proved in \cite{KKKO15} but it
can be proved similarly to
the quiver Hecke algebra case given in \cite[Theorem 2.11]{KKKO18}.

\smallskip
For modules $M$ and $N$ in $\uqm$, we denote by $M \hconv N$ and $M \sconv N$ the head and the socle of $M \otimes N$, respectively.

\section{New invariants for pairs of modules} \label{sec: New invs}
In this section, we introduce new invariants for
pairs of $\uqpg$-modules by using $R$-matrices and investigate their properties.
These invariants have similar properties to those in the quiver Hecke algebra case.

Recall that
$$ \text{$\tp \seteq p^{*2}=q^{2\ang{c,\rho}}$ and
$\vphi(z)=\prod_{s\in\Z_{\ge0}}(1-\tp^sz)\in\cor[[z]]$.}$$
We set
$$  \text{$\tp^{S} \seteq \{ \tp^k \ | \ k \in S\}$
\quad for a subset $S$ of $\Z$.}$$

\begin{definition}
We define the subset $\G$ of $\cor((z))^\times$ as follows:
\begin{align}\label{eq: mathcal G}
\G \seteq \left\{ cz^m \prod_{a \in \ko^\times} \varphi(az)^{\eta_a} \ \left|  \
\begin{matrix} \ c \in \ko^\times, \ m \in \Z , \\
\eta_a \in \Z \text{ vanishes except finitely many $a$'s. } \end{matrix} \right. \right\}.
\end{align}
\end{definition}
Note that $\G$ forms a group with respect to the multiplication. We have $\ko(z)^\times\subset \G$.
Note also that for $f(z)= cz^m \prod_{a \in \ko^\times} \varphi(az)^{\eta_a}$,
$\{\eta_a\}_{a\in\ko^\times}$ is determined by $f(z)$ since
\eqn
\dfrac{f(z)}{f(\tp z)}=(\tp)^{-m}\prod_{a\in\ko^\times}(1-az)^{\eta_a}.
\eneqn

\Prop
Let $M$ and $N$ be modules in $\uqm$.
If $\Runiv_{M,N_z}$ is \rr, then
the \rc $c_{M,N}(z)$ belongs to $\G$.
\enprop
\Proof
Let us take a simple submodule $M'$ of $M$ and
a simple submodule $N'$ of $N$.
Then, Proposition~\ref{prop:subqc} implies that
$c_{M,N}(z)/c_{M'N'}(z)\in\cor(z)^\times\subset\G$.
Hence the assertion follows from the following lemma.
\QED

\begin{lemma} For simple modules $M$ and $N$ in $\uqm$, the universal coefficient $a_{M,N}(z)$ as well as the \rc $c_{M,N}(z)$ is contained in $\G$.
\end{lemma}

\begin{proof}
Let us write $M=S((i_1,a_1),\ldots,(i_t,a_t))$ and $N=S((j_1,b_1),\ldots,(j_{t'},b_{t'}))$. When $t+t'=2$, $a_{M,N}(z)$ is nothing but
$a_{i,j}(b_1/a_1 z)$ in~\eqref{eq: aij}, and our assertion holds.
Then the induction on $t+t'$ proceeds
by Proposition~\ref{prop:rcomp} and  Proposition~\ref{prop:subqc}
\end{proof}

For each subset $S$ of $\Z$, we can construct a group homomorphism from $\Bg$
to the additive group $\Z$ by associating
the sum of  exponents $\eta_a$ such that $a \in \tp^S$.
For instance, by taking $S$ as $\Z$ or $\Z_{\le 0}$, we define the group homomorphisms
\begin{align*}
\tD \col   \Bg \to  \Z \quad \text{ and } \quad \Di \col   \Bg \to  \Z,
\end{align*} by
$$
\tD(f(z)) = \sum_{a \in \tp^{\,\Z_{\le 0}}}\eta_a \quad \text{ and } \quad \Di(f(z)) = \sum_{a \in \tp^{\,\Z}} \eta_a.
$$
for $f(z)=cz^m \prod \varphi(az)^{\eta_a} \in \Bg$. As their linear  combination,
we introduce the group homomorphism
\begin{align*}
\Deg \col   \Bg \to  \Z \quad \text{by} \quad  \Deg = 2\tD - \Di,
\end{align*}
namely,
\begin{align}\label{eq: Lambda def}
\Deg(f(z)) = \sum_{a \in \tp^{\,\Z_{\le 0}} }\eta_a - \sum_{a \in \tp^{\,\Z_{> 0}} } \eta_a.
\end{align}

Recall Convention~\ref{convention}~\eqref{conv:zero}.

\Lemma\label{lem:properties of Deg}
Let $f(z)\in\Bg$.
\bnum
\item \label{tdeg:1} If $f(z)\in \cor(z)$, then we have
$$\tD(f(z))=\zero_{z=1}f(z),\quad\Deg^\infty(f(z))=0,\text{\ and\;\ }
\Deg(f(z))=2\zero_{z=1}f(z).$$
\item If $g(z), \;h(z)\in\G$ satisfy
$g(z)/h(z)\in\cz$, then $\Deg(h(z))\le\Deg(g(z))$.
\label{degpol}
\item $\Deg^\infty f(z)=  -\Deg\bl f(\tp^{n}z)\br=\Deg\bl f(\tp^{-n}z)\br$
 for $n \gg 0$. \label{it:ninf}
\item If $\Deg^\infty\bl f(cz)\br=0$ for any $c\in\cor^\times$, then
$f(z)\in\cor(z)$.
\ee
\enlemma
\Proof
We may assume $f(z)=\prod_{a\in\cor^\times} \varphi(az)^{\eta_a}$.

\smallskip\noi
(i)\ For $a\in\cor^\times$, we have
\eqn\tD(1-az)&&=\tD\bl\vphi(az)/\vphi(\tp az)\br\\&&
=\delta(a\in\tp^{\Z\le0})-\delta(\tp a\in\tp^{\Z\le0})=\delta(a=1)=\zero_{z=1}(1-az)
\eneqn
and $$\Di(1-az)=\Di\bl\vphi(az)/\vphi(\tp az)\br
=1-1=0.$$

\smallskip\noi
(ii) follows from (i).

\smallskip\noi
(iii) We have
$$\Deg\bl f(\tp^{n}z)\br =  \sum_{a \tp^n \in \tp^{\Z \le 0}} \eta_a - \sum_{a\tp^n \in \tp^{\Z > 0}} \eta_a.$$
Hence we have $\Deg\bl f(\tp^{n}z)\br=-\sum_{a\in \tp^{\Z}} \eta_a$ if $n\gg0$ and
$\Deg\bl f(\tp^{n}z)\br=\sum_{a\in \tp^{\Z}} \eta_a$ if $n\ll0$.

\smallskip\noi
(iv)
By the assumption, we can easily see that
$f(z)$ is a product of functions of the form
$\vphi(az)/\vphi(\tp^maz)$ ($a\in\cor^\times$, $m\in\Z$).
Then the result follows from
$\vphi(az)/\vphi(\tp^maz)\in\cor(z)$.
\QED

\Rem
Any $f(z)\in\G$
extends to a meromorphic function
on $$\st{(z,q^{1/\ell})\in\C\times\C\;;\;|q^{1/\ell}|<\eps}$$
 for some $\ell\in\Z_{>0}$
and $\eps>0$.
Hence $\zero_{z=\tp^k}f(z)$, the order of zero of $f(z)$ at $z=\tp^k$,
 makes sense for any $k\in\Z$.
Then one has
$\tD (f(z))=\zero_{z=1}f(z)$.
\enrem

Using the homomorphisms $\Deg$, $\tD$ and $\Di$, we define the new invariants for a pair of modules $M$, $N$ in $\uqm$ such that $\Runiv_{M,N_z}$ is \rr.

\begin{definition} \label{def: Lams}
For non-zero modules $M$ and $N$ in $\uqm$ such that $\Runiv_{M,N_z}$ is \rr, we define the integers $\Lambda(M,N)$, $\tLa(M,N)$ and $\Lambda^\infty(M,N)$ as follows:
$$\  \Lambda(M,N)=\Deg(c_{M,N}(z)),\ \tLa(M,N)=\tD(c_{M,N}(z)), \ \Lambda^\infty(M,N)=\Deg^\infty(c_{M,N}(z)).$$
\end{definition}
Hence, we have
\eq\tLa(M,N)=\dfrac{1}{2}\Bigl(\La(M,N)+\Lambda^\infty(M,N)\Bigr).
\label{eq:LtL}\eneq

\Lemma\label{lem:Lstar}
For any simple modules $M$, $N$ in $\uqm$ and $x\in\cor^\times$, we have
\eqn
&&\La(M,N)=\La(M^*,N^*)=\La(\rd M,\rd N)=\La(M_x,N_x),\\
&&\tL(M,N)=\tL(M^*,N^*)=\tL(\rd M,\rd N)=\tL(M_x,N_x),\\
&&\Li(M,N)=\Li(M^*,N^*)=\Li(\rd M,\rd N)=\Li(M_x,N_x).
\eneqn
\enlemma
\Proof
They follow from Proposition~\ref{prop:cstar}.
\QED

\begin{lemma} \label{lem: properties of La}
Let $M$ and $N$ be non-zero modules in $\uqm$.
\bnum
\item\label{di0} If $M$ and $N$ are simple, then we have
$\Lambda^\infty(M,N)  = \Deg^\infty(c_{M,N}(z)) =-\Deg^\infty(a_{M,N}(z))$.
\item  If $\Runiv_{M,N_z}$ is \rr, then
\begin{align*}
\Lambda^\infty(M,N) & =  -\Lambda(M,N_{\tp^{n}})=\Lambda(M,N_{\tp^{-n}}) \quad \text{ for $n \gg 0$}.
\end{align*}
\end{enumerate}
\end{lemma}

\begin{proof}
(i) follows from $a_{M,N}(z)c_{M,N}(z)\in\cor(z)$ and Lemma~\ref{lem:properties of Deg} \eqref{degpol}.

\noindent \noi
(ii) follows from $c_{M,N_{\tp^n}}(z)=c_{M,N}(\tp^nz)$
and Lemma~\ref{lem:properties of Deg} \eqref{it:ninf}.
\end{proof}

\Prop\label{prop:subqcL}
Let $M$ and $N$ be modules in $\uqm$, and let
$M'$ and $N'$ be a non-zero subquotient of $M$ and $N$, respectively.
Assume that $\Runiv_{M,N_z}$ is \rr.
Then $\Runiv_{M',N'_z}$ is \rr, and
$$\La(M',N')\le \La(M,N)\qtq \Li(M',N')=\Li(M,N).$$
\enprop
\Proof
They follow from Proposition~\ref{prop:subqc} and Lemma~\ref{lem:properties of Deg}.
\QED

\begin{lemma} \label{lem: Lambda tensor} Let $M$, $N$ and $L$
be non-zero modules in $\uqm$.
\bnum
\item
If $\Runiv_{M_,\,L_z}$ and $\Runiv_{N,\,L_z}$ are \rr,
then $\Runiv_{M\tens N_,\,L_z}$ is \rr and
$$ \La(M \otimes N,L)\le  \La(M,L)+\La(N,L)\ \text{and}\
\Li(M \otimes N,L)=\Li(M,L)+\Li(N,L).
$$
If we assume further that $L$ is simple, then the equality holds
instead of the inequality.

\vs{1ex}
\item
If $\Runiv_{L_,M_z}$ and $\Runiv_{L,N_z}$ are \rr,
then $\Runiv_{L,\;(M\tens N)_z}$ is \rr and
$$ \La(L,M \otimes N)\le  \La(L,M)+\La(L,N)
\ \text{and}\ \Li(L,M \otimes N)=\Li(L,M)+\Li(L,N).$$
If we assume further that $L$ is simple, then the equality holds
instead of the inequality.
\ee
\end{lemma}
\Proof
They follow from Proposition~\ref{prop:rcomp}.
\QED

\begin{proposition}\label{prop:hersubq}
Let $M$, $N$ and $L$ be non-zero modules in $\uqm$,
and let $S$ be a non-zero subquotient of $M\tens N$.
\bnum
\item
Assume that $\Runiv_{M,L_z}$ and $\Runiv_{N,L_z}$ are \rr.
Then $\Runiv_{S,L_z}$ is \rr and
$$ \Lambda(S,L)\le  \Lambda(M,L) + \Lambda(N,L)
\qtq
\Li(S,L)=\Li(M,L) + \Li(N,L).
$$
\item
Assume that $\Runiv_{L,M_z}$ and $\Runiv_{L,N_z}$ are \rr.
Then $\Runiv_{L,S_z}$ is \rr and
$$\Lambda(L,S)\le\Lambda(L,M) + \Lambda(L,N)
\qtq \Li(L,S)=\Li(L,M) + \Li(L,N).$$
\ee
\end{proposition}

\begin{proof}
These assertions follow from Propositions~\ref{prop:subqcL} and
\ref{lem: properties of La}.
\end{proof}

\begin{corollary} For simple modules
$$M=S\big((i_1,a_1),\ldots, (i_\ell,a_\ell)\big) \quad \text{and}
\quad N=S\big((j_1,b_1),\ldots, (j_{\ell'},b_{\ell'}) \big) \quad \text{ in } \uqm,$$ we have
$$ \Lambda^{\infty}(M,N) = \sum_{ 1 \le \nu \le \ell, \  1 \le \mu \le \ell'} \Lambda^{\infty}(V(\varpi_{i_\nu})_{a_{\nu}},V(\varpi_{j_\mu})_{b_{\mu}}).$$
\end{corollary}

\begin{example}
Take $L=M=V(\varpi_1)_{(-q)^{-2}}$ and $N=V(\varpi_1)$ over $U'_q(A^{(1)}_{2})$ where $p^*=(-q)^3$ and $\tp=q^6$. Then we have
\begin{align*}
&c_{M,L}(z)=\dfrac{[2][-2]}{[0][6]},\  c_{N,L}(z)=\dfrac{[0][-4]}{[-2][4]}\qt{and hence} c_{M,L}(z) c_{N,L}(z)=\dfrac{[2][-4]}{[6][4]}.
\end{align*}
On the other hand, we have $M\hconv N=V(\varpi_2)_{(-q)^{-1}}$ and
$$ c_{M\tiny{\hconv} N,L}(z)=\dfrac{[2][-4]}{[0][4]}.$$
Thus we have
$$ \tLa(M,L)+\tLa(N,L)=(-1)+1=0, \qquad \tLa(M \hconv N,L)=-1 $$
and hence
$$  \La(M,L)+\La(N,L)-\La(M \hconv N,L)=2 \qt{and}  c_{M\tiny{\hconv} N,L}(z) \times (1-z) =  c_{M,L}(z) c_{N,L}(z).$$
\end{example}

\begin{definition}[{see\ Corollary \ref{cor:desym}}]
For simple modules $M$ and $N$ in $\uqm$, we define $\de(M,N)$ by
$$ \de(M,N)= \dfrac{1}{2}\bl\Lambda(M,N) + \Lambda(M^*,N)\br.$$
\end{definition}

Now we will prove that $ \de(M,N)$ is non-negative integer. In order to do that, we need some preparation.

\begin{lemma} \label{lem: computation a c}
For simple modules $M$ and $N$ in $\uqm$, we have
$$  c_{M,N}(z)c_{M^*,N}(z) \equiv d_{M,N}(z)d_{N,M}(z^{-1}) $$  and
$$\dfrac{c_{M,N}(z)}{c_{M,N}(\tp z)} \equiv \dfrac{d_{M,N}(z)d_{N,M}(z^{-1})}{d_{M^*,N}(z)d_{N,M^*}(z^{-1})}$$
up to a multiple of $\ko[z^{\pm1}]^\times$.
\end{lemma}

\begin{proof}
By Proposition~\ref{prop: aMN} \eqref{MN2}, we have
$$    a_{M^*,N}(z)a_{M,N}(z) \equiv \dfrac{d_{M^*,N}(z)}{d_{N,M}(z^{-1})} \qquad {\rm mod} \  \ko[z^{\pm1}]^\times.$$
Recall that $c_{M,N}(z) =\dfrac{d_{M,N}(z)}{a_{M,N}(z)}$.  Then we have
\begin{align*}
c_{M,N}(z)c_{M^*,N}(z) &= \dfrac{d_{M,N}(z)}{a_{M,N}(z)} \times \dfrac{d_{M^*,N}(z)}{a_{M^*,N}(z)} \\
 &\equiv d_{M,N}(z) \times  d_{M^*,N}(z) \times \dfrac{d_{N,M}(z^{-1})}{d_{M^*,N}(z)} \\
 &\equiv d_{M,N}(z) \times  d_{N,M}(z^{-1})
\qquad {\rm mod} \  \ko[z^{\pm1}]^\times.
\end{align*}
Thus we have
\[
\dfrac{c_{M,N}(z)}{c_{M,N}(\tp z)} = \dfrac{c_{M,N}(z)c_{M^*,N}(z)}{c_{M^*,N}(z)c_{M^{**},N}(z)} \equiv
\dfrac{d_{M,N}(z) d_{N,M}(z^{-1})}{d_{M^*,N}(z) d_{N,M^*}(z^{-1})} \qquad {\rm mod} \  \ko[z^{\pm1}]^\times. \qedhere
\]
\end{proof}

\begin{proposition} \label{prop: de(M,N)}
For simple modules $M$ and $N$ in $\uqm$, we have
\eq
\de(M,N)=\zero_{z=1}\bl d_{M,N}(z)d_{N,M}(z^{-1})\br.\label{eq:deMN}
\eneq
In particular,
\[
 \de(M,N)\in\Z_{\ge0},
\] and
\begin{align}\label{eq:symmetry}
\de(M,N) = \de(N,M).
\end{align}
\end{proposition}

\begin{proof}
By the preceding lemma,
\eqn
2\de(M,N)=\Deg\bl c_{M,N}(z)c_{M^*,N}(z)\br
&=&\Deg\bl d_{M,N}(z)d_{N,M}(z^{-1})\br\\
&=&2\zero_{z=1}\bl d_{M,N}(z)d_{N,M}(z^{-1})\br.
\eneqn
Here the last equality follows from
Lemma~\ref{lem:properties of Deg}\;\eqref{degpol}.
The other assertions follow from \eqref{eq:deMN}.
\end{proof}

\Cor\label{cor:commde}
Let $M$ and $N$ be simple modules in $\uqm$.
Assume that one of them is real.
Then $M$ and $N$ strongly commute if and only if
$\de(M,N)=0$.
\encor
\Proof
It follows from Proposition~\ref{prop: de(M,N)} and Theorem~\ref{Thm: basic properties}~
\eqref{it:comm}.
\QED

For $k \in \Z$ and a module $M$ in $\uqm$, we define
$$
\D^k(M) \seteq
\begin{cases}
(\cdots( M^* \underbrace{ )^* \cdots )^* }_{\text{$(-k)$-times}} & \text{ if } k <0, \\
 \underbrace{\rd  ( \cdots ( }_{\text{$k$-times}} \rd M  )\cdots)  & \text{ if } k \ge 0.
\end{cases}
$$

\begin{proposition} \label{prop: sym lam}
For simples $M$ and $N$ in $\uqm$, we have
$$\La(M,N)=\La(N^*,M)=\La(N,\rd M).$$
\end{proposition}

\begin{proof}
We shall prove $\La(M,N)=\La(N^*,M)$. The other equality follows from  Lemma~\ref{lem:Lstar}.

By~\eqref{eq:symmetry}, we have
\begin{align*}
& \La(M,N)+\La(M^*,N) = \La(N,M)+\La(N^*,M) \\
& \hspace{30ex} \iff  \La(M,N)-\La(N^*,M) = \La(N,M)-\La(M^*,N).
\end{align*}
Set $$K(M,N) \seteq \La(M,N)-\La(N^*,M).$$ Then we have $K(M,N)=K(N,M)$ and
$$ K(M^*,N)=\La(M^*,N)-\La(N^*,M^*)\underset{(\star)}{=}\La(M^*,N)-\La(N,M)= -K(N,M)
=-K(M,N),$$
where ($\star$) follows from Lemma~\ref{lem:Lstar}.
Hence we have
$$ K(M,N) = K(\D^{2n}(M),N) \qquad \text{ for any $n \in \Z$}.$$

Note that, for $n \gg 0$, we have
\begin{subequations} \label{eq: behaviour}
\begin{align}
& \La(\D^{2n}(M),N)=\La(M_{\tp^{n}},N)=\La(M,N_{\tp^{-n}})
=\bc \Li(M,N)&\text{if $n\gg0$,}\\
-\Li(M,N)&\text{if $n\ll0$,}\ec\\
& \La(N^*,\D^{2n}(M)) =\La(N^*,M_{\tp^{n}})
=\bc -\Li(N^*,M)&\text{if $n\gg0$,}\\
\Li(N^*,M)&\text{if $n\ll0$.}\ec
\end{align}
\end{subequations}
Thus, for $n \gg 0$, we have
$K(\D^{2n}(M),N) = - K(\D^{-2n}(M),N)$, which implies
$K(M,N)=-K(M,N)$.
Finally, we conclude that
\[ K(M,N) =0 .  \qedhere \]
\end{proof}

\Cor \label{cor:desym}
For any simple modules $M$ and $N$ in $\uqm$, we have
$$\de(M,N)=\dfrac{1}{2}\Bigl(\La(M,N)+\La(N,M)\Bigr).$$
\encor
\begin{corollary}\label{cor:real0}
For any real simple $M$ in $\uqm$, we have
$$ \La(M,M)=0.$$
\end{corollary}

\begin{proof}
By Corollary~\ref{cor:commde}, Corollary~\ref{cor:desym} and the assumption that $M$ is real simple, we have $$0=2\de(M,M)=\La(M,M)+\La(M,M),$$
which implies our assertion.
\end{proof}

\Rem\label{rem:dual}
The formula in Proposition~\ref{prop: sym lam}
holds also for objects in the rigid monoidal category $\widetilde{\shc}_w$ (see
\cite{KKOP19A}).
Indeed, we have
$$\HOM (N^*\tens M_z,M_z\tens N^*)\simeq \HOM(M_z\tens N,N\tens M_z)$$
and hence their generators $\Rnorm_{N^*, M_z}$ and $\Rnorm_{M_z,N}$ have the same homogeneous degree.
\enrem

\begin{proposition} \label{prop: de and Lambda}
For simple modules $M$ and $N$ in $\uqm$,  we have the followings:
\begin{enumerate}[{\rm (i)}]
\item $\Lambda(M,N)= \sum_{k \in \Z} (-1)^{k+\delta(k<0)} \de(M,\D^{k}N)$,
\item $\Lambda^\infty(M,N)= \sum_{k \in \Z} (-1)^{k} \de(M,\D^{k}N)$,
\item
\berm $\zero_{z=1}c_{M,N}(z)=\displaystyle\sum_{k=0}^\infty(-1)^k\de(M,\D^kN)$. \er
\end{enumerate}
\end{proposition}

\begin{proof}
Write $c_{M,N}(z) \equiv \prod \varphi(az)^{\eta_a} \mod \ \ko[z^{\pm1}]^\times$. Then we have
$$ \dfrac{c_{M,N}(z)}{c_{M,N}(\tp z)} \equiv \prod (1-az)^{\eta_a}.$$
and hence
\berm
\begin{align*}
\eta_{\tp^{-k}} & =\zero_{z= \tp^{k}}\left(\dfrac{c_{M,N}(z)}{c_{M,N}(\tp z)}\right)
=  \zero_{z= 1}\left(\dfrac{c_{M,N}(\tp^{k}z)}{c_{M,N}(\tp^{k+1} z)}\right) 
=  \zero_{z= 1}\left(\dfrac{c_{M,N_{\tp_k}}(z)}{c_{M,N_{\tp^{k}}}(\tp z)}\right) \\
& \underset{(*)}{=}
 \dfrac{d_{M,N_{\tp_k}}(z)d_{N_{\tp_k},M}(z^{-1})}
{d_{M^*,N_{\tp^{k}}}(z)d_{N_{\tp^{k}},M^*}(z^{-1})}
 \underset{(**)}{=}
\de(M,N_{\tp^{k}}) - \de(M^*,N_{\tp^{k}}) \\
&  = \de(M,\D^{2k} N) - \de(M,\D^{2k+1}N).
\end{align*}
Here $ (*)$ follows from Lemma~\ref{lem: computation a c} and 
$(**)$ from Proposition~\ref{prop: de(M,N)}.

Thus we have
\begin{align*}
\Lambda(M,N) & =\sum_{k \in \Z} (-1)^{\delta(k>0)} \eta_{\tp^k} 
 =\sum_{k \in \Z} (-1)^{\delta(k<0)} \eta_{\tp^{-k}} \\
&=\sum_{k \in \Z} (-1)^{\delta(k<0)} (\de(M,\D^{2k}N )- \de(M^*,\D^{2k+1}N)) \\
&= \sum_{k \in \Z} (-1)^{k+\delta(k<0)} \de(M,\D^{k}N ),
\end{align*}\er
which imply the first assertion.
Similarly, we have
\begin{align*}
\Lambda^\infty(M,N) & = \sum_{k \in \Z}     (\de(M,\D^{-2k}N )-\de(M,\D^{-2k+1}N)) = \sum_{k \in \Z} (-1)^{k}  \de(M,\D^{k}N ) 
\end{align*}
\berm Fimally. we have
\eqn
\zero_{z=1}c_{M,N}(z)=\sum_{k=0}^\infty\eta_{\tp^{-k}}
=\sum_{k=0}^\infty\bl\de(M,\D^{2k} N) - \de(M,\D^{2k+1}N)\br
=\sum _{k=0}^\infty(-1)^k\de(M,\D^{k} N).
\eneqn\qedhere \er
\end{proof}

The following corollary is a direct consequence of Proposition~\ref{prop: de and Lambda} and~\eqref{eq:symmetry}:

\begin{corollary} \label{cor:property Lainf} For simple modules $M$ and $N$ in $\uqm$, we have
\begin{enumerate}
\item $ \Lambda^{\infty}(M,N) = \Lambda^{\infty}(N,M).$
\item $ \Lambda^{\infty}(M,N) = -\Lambda^{\infty}(M^*,N)= -\Lambda^{\infty}(\rd M,N).$
\end{enumerate}
\end{corollary}

\begin{proof}
Sine $\de(M,N)=\de(\D^kM,\D^kN)$, we have
\begin{align*}
\Lambda^\infty(M,N) &= \sum_{k \in \Z} (-1)^{k}  \de(M,\D^{k}N ) = \sum_{k \in \Z} (-1)^{k}  \de(\D^{k}M,N ) \\
&  =  \sum_{k \in \Z} (-1)^{k}  \de(N,\D^{k}M)  = \Lambda^\infty(N,M).
\end{align*}
Hence the first assertion follows. The second assertion follows similarly.
\end{proof}

\section{Further properties of the invariants} \label{sec: Further properties}
We start this section with the following proposition, which can be understood as a quantum affine analogue of
\cite[Proposition 3.2.8]{KKKO18}:

\begin{proposition} \label{prop:inequalities}
Let $N_1$, $N_2$ and $M$ be non-zero modules in $\Ca_\g$ and let
$f\col N_1 \to N_2$ be a morphism. We assume that $\Runiv_{N_k,M_z}$ is \rr
for $k=1,2$.
\begin{enumerate}[{\rm (i)}]
\item
If $f$ does not vanish, then $c_{N_1,M}(z)/c_{N_2,M}(z)\in\cor(z)$ and
$$\Lambda(N_1,M) - \Lambda(N_2,M)=
2\zero_{z=1}\left(\dfrac{c_{N_1,M}(z)}{c_{N_2,M}(z)}\right).$$
\item
If $\Lambda(M,N_1)=\Lambda(M,N_2)$, then the following diagram is commutative:
$$
\xymatrix@C=5em{
N_1 \otimes M \ar[r]^{\rmat{N_1,M}} \ar[d]_{f \tens M} &
M \otimes N_1 \ar[d]^{M \tens f} \\
N_2 \otimes M \ar[r]^{\rmat{N_2,M}}&
M \otimes N_2.
}
$$\label{it:commr}
\item If $\Lambda(N_1,M) > \Lambda(N_2,M)$, then the composition
$$ N_1 \otimes M \To[\rmat{N_1,M}] M \otimes N_1 \To[M \circ f] M \otimes N_2 $$
  vanishes.
 \item If $\Lambda(N_1,M) < \Lambda(N_2,M)$, then the composition
 $$   N_1 \otimes M \To[f \circ M] N_2 \otimes M \To[\rmat{N_2,M}] M \otimes N_2$$
  vanishes.
\end{enumerate}
\end{proposition}
Although we don't write, similar statements hold for $c_{M,N_k}(z)$
and $M\tens N_k$.
\begin{proof}
Without loss of generality,
we may assume that $f$ is non-zero.

\smallskip
\noi
(i)  Proposition~\ref{prop:subqc} implies that
$\dfrac{c_{N_1,M}(z)}{c_{N_2,M}(z)} \in \ko(z)$.
Hence we have by Lemma~\ref{lem:properties of Deg}
$$2\zero_{z=1}\left(\dfrac{c_{N_1,M}(z)}{c_{N_2,M}(z)}\right)=
\Deg\left(\dfrac{c_{N_1,M}(z)}{c_{N_2,M}(z)}\right) =
\Lambda(N_1,M) - \Lambda(N_2,M).$$

Set $t=\zero_{z=1}\bl{c_{N_1,M}(z)}/{c_{N_2,M}(z)}\br$.
Then we can write
$g(z)c_{N_1,M}(z)=h(z)(z-1)^t c_{N_2,M}(z)$
for some $t\in\Z$ and $g(z),h(z)\in\cor[z]$ which do not vanish at $z=1$.

\medskip
If $t\ge0$, then we have the following commutative diagram
\begin{align} \label{eq: com diam}
  \xymatrix@C=14em{
  N_1 \otimes M_z \ar[r]^{  g(z) \Rren_{N_1,M_z}} \ar[d]_{f \otimes M_z} &  M_z \otimes N_1 \ar[d]^{M_z \otimes f} \\
  N_2 \otimes M_z \ar[r]^{h(z) (z-1)^{t} \Rren_{N_2,M_z}}&  M_z \otimes N_2.
  }
\end{align}

(i) Since $t=0$, by specializing $z=1$ in the above diagram,
we obtain the commutativity of~\eqref{eq: com diam}.

(ii) Since $t > 0$, the homomorphism
$h(z)(z-1)^{t} \Rren_{N_2,M_z}$ vanishes at $z=1$.
Hence we have
$$(M \otimes f) \circ \rmat{N_1,M} = (z-1)^{t} \Rren_{N_2,M_z}\big|_{z=1} \circ (f\otimes M) =0,$$
as desired.

(iii) Since $t<0$, we have the following commutative diagram
\begin{align}
  \xymatrix@C=14em{
  N_1 \otimes M_z \ar[r]^{g(z)(z-1)^{-t} \Rren_{N_1,M_z}} \ar[d]_{f \otimes M_z} &  M_z \otimes N_1 \ar[d]^{M_z \otimes f} \\
  N_2 \otimes M_z \ar[r]^{h(z)\Rren_{N_2,M_z}}&  M_z \otimes N_2.
  }
\end{align}
Since $g(z)(z-1)^{-t} \Rren_{N_1,M_z}$ vanishes at $z=1$,
we obtain the desired result.
\end{proof}

From the above proposition,
we can show that the new invariants share
similar properties with the one for quiver Hecke algebras
studied in \cite[Section 3.2]{KKKO18}.
We will collect such properties.
Since the proofs are similar, we sometimes omit the proofs.

\begin{proposition} \label{pro:subquotient}
Let $L$, $M$ and $N$ be simple modules.
Then we have
\begin{equation} \label{eq:LMN<}
\begin{aligned}
&\de(S,L)\le\de(M,L)+\de(N,L)
\end{aligned}
\end{equation}
for any simple subquotient $S$ of $M\otimes N$.
Moreover, when $L$ is real, the following conditions are equivalent.
\begin{enumerate}
\item[{\rm (a)}] $L$ strongly commutes with $M$ and $N$.
\item[{\rm (b)}] Any simple subquotient $S$ of $M\otimes N$ commutes with $L$ and satisfies
$\La(S,L)=\La(M,L)+\La(N,L)$.
\item[{\rm (c)}] Any simple subquotient $S$ of $M\otimes N$ commutes with
$L$ and satisfies $\La(L,S)=\La(L,M)+\La(L,N)$.
\end{enumerate}
\end{proposition}

\begin{lemma} \label{lem:LMN1}
Let $L$, $M$ and  $N$ be simple modules in $\uqm$, and assume that $L$ is real.
\begin{enumerate}[{\rm (i)}]
\item If $L$ strongly commutes with $N$, then the diagram
$$\xymatrix@C=10ex
{(M\otimes N)\otimes L\ar[r]^{\rmat{M\tens N,L}}\ar[d]
&L\otimes(M\otimes N)\ar[d]\\
(M\hconv N)\otimes L\ar[r]^{\rmat{M\tiny{\hconv} N,L}}
&L\otimes(M\hconv N)}
$$
commutes and
$$\La(M\hconv N,L)=\La(M,L)+\La(N,L).$$
\item If $L$ strongly commutes with $M$, then the diagram
$$\xymatrix@C=10ex
{L\otimes (M\otimes N)\ar[r]^{\rmat{L,M\circ N}}\ar[d]
&(M\otimes N)\otimes L\ar[d]\\
L\otimes (M\hconv N)\ar[r]^{\rmat{L,M\tiny{\hconv} N}}
&(M \hconv N)\otimes L}
$$
commutes and
$$\La(L,M\hconv N)=\La(L,M)+\La(L,N).$$
\end{enumerate}
\end{lemma}

\Cor\label{cor:aadd}
Let $L$, $M$ $N$ be non-zero modules in $\uqm$.
Assume that $L$ is real.
Then we have
\bnum
\item If $L^*$ and $M$ strongly commute, then
$$\La(M\hconv N,L)=\La(M,L)+\La(N,L).$$
\item If $L$ and $N^*$ strongly commute, then
$$\La(L, M\hconv N)=\La(L,M)+\La(L,N).$$
\ee
\encor
\Proof
(i) We have
\eqn
&&\La(M\hconv N,L)\underset{(*)}{=}\La(L^*, M\hconv N)\underset{(**)}{=}\La(L^*,M)+\La(L^*,N)=\La(M,L)+\La(N,L),
\eneqn
where $(*)$ follows from Proposition~\ref{prop: sym lam}
and $(**)$ from Lemma~\ref{lem:LMN1}.
The proof of (ii) is similar.
\[
 \La(L, M\hconv N)=\La(M\hconv N,\rd L)=\La(M,\rd L)+\La(N,\rd L)=\La(L,M)+\La(L,N). \qedhere
\]
\QED

\begin{proposition} \label{prop:sochdNM}
Let $M$ and $N$ be non-zero modules in $\uqm$ and assume that
$M$ is real.
\begin{enumerate} [{\rm (i)}]
\item Assume that  $N$ has a simple socle,
$\Runiv_{N,M}$ is \rr and the diagram
$$\xymatrix@C=12ex
{\wb{\soc(N)\otimes M}\ar[r]^{\rmat{\soc(N),M}}\ar@{>->}[d]
&\wb{M\otimes\soc(N)}\ar@{>->}[d]\\
N\otimes M\ar[r]^{\rmat{N,M}}&M\otimes N
}
$$ commutes up to a non-zero constant multiple.
Then
$ M\sconv\soc(N)$ is isomorphic to
the socle of $M\otimes N$.
In particular, $M\otimes N$ has a simple socle.
\item  Assume that $N$ has a simple head,
$\Runiv_{M,N}$ is \rr and the diagram
$$
\xymatrix@C=10ex@R=6ex{
M \otimes N \ar[rr]^{\rmat{M,N}} \ar@{->>}[d] && N \otimes M  \ar@{->>}[d] \\
M \otimes \hd(N)\ar[rr]^{\rmat{M,\hd(N)}} && \hd(N) \otimes M
}
$$
commutes up to a non-zero constant multiple, then
$M\hconv \hd(N)$ is equal to the simple head of $M\otimes N$
\end{enumerate}
\end{proposition}

\begin{proof}
Let $S$ be an arbitrary simple submodule of $M\otimes N$.
Then we have the following commutative diagram:
$$\xymatrix@C=12ex
{\wb{S\otimes M_z}\ar[r]^{ f(z)(z-1)^m\Rren_{S\otimes M_z} }  \ar@{>->}[d]
&\wb{M_z\otimes S}\ar@{>->}[d]\\
M\otimes N\otimes M_z\ar[r]^{\Rren_{M\otimes N,M_z}}
&M_z\otimes M\otimes N.
}$$
for some $f(z)\in \cor(z)$ which is regular and do not vanish at $z=1$ and $m \in \Z_{\ge 0}$. By specializing at $z=1$,
we have a commutative diagram  (up to a constant multiple):
$$\xymatrix@C=12ex
{\wb{S\otimes M}\ar[r] \ar@{>->}[d]
&\wb{M\otimes S}\ar@{>->}[d]\\
M\otimes N\otimes M\ar[r]^{M\otimes\rmat{N,M}}&M\otimes M\otimes N.
}$$

Here, we use the fact that
$\rmat{M\otimes N, M} = (\rmat{M,M} \otimes N) \circ (M\otimes \rmat{N,M})$ and
$\rmat{M,M}$ is equal to ${\rm id}_{M\otimes M}$ up to a non-zero constant multiple, because $M$ is real.

It follows that  $S\otimes M\subset M\otimes (\rmat{N,M})^{-1}(S)$.
Hence there exists a submodule $K$ of $N$ such that
$S\subset M\otimes K$ and $K\otimes M\subset (\rmat{N,M})^{-1}(S)$
by Lemma \ref{lem: middle}.
Hence $K\not=0$ and $\soc(N)\subset K$ by the assumption.
Hence $\rmat{N,M}\bl\soc(N)\otimes M\br
\subset \rmat{N,M}\bl K\otimes M\br\subset S$.
Since $\rmat{N,M}\bl\soc(N)\otimes M\br$ is non-zero by the assumption, we have
$\rmat{N,M}\bl\soc(N)\otimes M\br=S$.
Thus we obtain the desired result for the first assertion.

The second assertion can be proved similarly.
\end{proof}

\begin{proposition} \label{prop:3simple}
Let $L$, $M$ and  $N$ be simple modules.
We assume that $L$ is real and
one of $M$ and $N$ is real.
\begin{enumerate}
\item[{\rm (i)}]
If $\La(L,M\hconv N)=\La(L,M)+\La(L,N)$,
then
$L\otimes M\otimes N$ has a simple head and
$N\otimes M\otimes L$ has a simple socle.
\item[{\rm (ii)}]
If $\La(M\hconv N,L)=\La(M,L)+\La(N,L)$,
then
$M\otimes N\otimes L$ has a simple head and
$L\otimes N\otimes M$ has a simple socle.
\item[{\rm (iii)}]
If $\de(L,M\hconv N)=\de(L,M)+\de(L,N)$,
then
$L\otimes M\otimes N$ and $M\otimes N\otimes L$  have  simple heads, and
$N\otimes M\otimes L$ and $L\otimes N\otimes M$ have  simple socles.
\end{enumerate}
\end{proposition}

\begin{proposition}\label{Prop: l2}
Let $M$ and $N$ be simple modules.
Assume that one of them is real and $\de(M,N)=1$.
Then we have an exact sequence
$$0\to  M\sconv N \to M\tens N\to M\hconv N\to 0.$$
In particular, $M\tens N$ has length $2$.
\end{proposition}

\begin{proof}
By Theorem~\ref{thm: KKKo15 main}, Proposition~\ref{prop: de(M,N)}
and \eqref{eq:RR}, we can apply the same argument in the proof
of \cite[Lemma 7.3]{KO18}.
\end{proof}

\begin{definition}
For simple modules $M$ and $M'$ in $\uqm$, we say that they are \emph{simply linked} if $ \de(M,M')=1$.
\end{definition}

\begin{proposition} \label{prop:real}
Let $X,Y,M$ and $N$ be simple modules in $\uqm$.
 Assume that there is an exact sequence
$$0 \to X \to M \tens N \to Y \to 0, $$
and $X \tens N$ and $Y \tens N$ are simple.
\bnum
\item \label{item: real i} If  $X \tens N \not\simeq Y \tens N$, then $N$ is a real simple module.
\item \label{item: real ii} If $M$ is real, then $N$ is a real simple module.
\end{enumerate}
\end{proposition}

\begin{lemma}\label{lem:MN}
Let $\{M_i\}_{1\le i\le n}$ and $\{N_i\}_{1\le i\le n}$ be a pair of
commuting families of real simple modules in $\uqm$.
We assume that
\begin{enumerate}[{\rm (a)}]
 \item $\{M_i\hconv N_i\}_{1\le i\le n}$ is a commuting family of real simple modules,
\item $M_i\hconv N_i$ commutes with $N_{j}$ for any $1\le i,j\le n$.
\end{enumerate}
Then we have
$$\left(\sotimes_{1\le i\le n}M_i\right)\hconv \left(\sotimes_{1\le j\le n}N_j\right)
\simeq \sotimes_{1\le i\le n}\left(M_i\hconv N_i \right).
$$
\end{lemma}

\begin{theorem}\label{th:leclerc}
Let $M$ and $N$ be simple modules.
We assume that $M$ is real.
Then we have
the equalities in the Grothendieck group $K(\uqm)${\rm:}
\begin{enumerate}[{\rm (i)}]
\item
$[M\tens N]=[M\hconv N]+\sum_{k}[S_k]$\\[.5ex]
with simple modules $S_k$ such that $\La(M,S_k)<\La(M,M\hconv N)=\La(M,N)$,
\item
$[M\tens N]=[ M\sconv N]+\sum_{k}[S_k]$\\[.5ex]
with simple modules $S_k$ such that $\La(S_k,M)<\La( M\sconv N,M)=\La(N,M)$,
\item
$[N\tens M]=[N\hconv M]+\sum_{k}[S_k]$\\[.5ex]
with simple modules $S_k$ such that $\La(S_k,M)<\La(N\hconv M,M)=\La(N,M)$,
\item
$[N\tens M]=[ N\sconv M]+\sum_{k}[S_k]$\\[.5ex]
with simple modules $S_k$ such that $\La(M,S_k)<\La(M, N \sconv M )=\La(M,N)$.
\end{enumerate}
In particular, $M\hconv N$ as well as $ M \sconv N $ appears only once
in the Jordan-H\"older series of $M\tens N$ in $\uqm$.
\end{theorem}

\begin{proof}
We shall prove only (iii).
The other statements are proved similarly.
First remark that
$\La(N\hconv M,M)=\La(N,M)+\La(M,M)=\La(N,M)$ by
Lemma~\ref{lem:LMN1} and
Corollary~\ref{cor:real0}.

Let
$$N\otimes M=K_0\supset K_1\supset \cdots\supset K_\ell\supset K_{\ell+1}=0$$
be a Jordan-H\"older series of $N\otimes M$. Then we have $K_0/K_1\simeq N\hconv M$.
Let us consider the
renormalized R-matrix 
$\Rren_{N \otimes M,M_z}=(\Rren_{N,M_z}\otimes M)\circ(N\otimes \Rren_{M,M_z})$
$$\xymatrix@C=10ex{
N\otimes M \otimes M_z \ar[r]^{N \otimes \Rren_{M,M_z}}
&N\otimes M_z\otimes M\ar[r]^{\Rren_{N,M_z} \otimes M} & M_z\otimes N\otimes M.}
$$

Then $\Rren_{N\otimes M,M_z}$ sends $K_k\otimes M_z$ to $M_z\otimes K_k$ for any $k$.
By evaluating the above diagram at $z=1$, we obtain
$$\xymatrix@C=10ex{
N\otimes M\otimes M\ar[r]^{\rmat{N,M} \otimes M}
&M\otimes N\otimes M\\
\wb{K_1\otimes M}\ar[r]\ar@{^{(}->}[u] &\wb{M \otimes K_1.}\ar@{^{(}->}[u]}
$$

Since $\Im(\rmat{N,M}\colon N\otimes M\to M\otimes N )\simeq (N\otimes M)/K_1$,
we have
$\rmat{M,N}(K_1)=0$.
Hence, $\Rren_{N\otimes M,M_z}$ sends $K_1 \otimes M_z$ to
$(M_z \otimes K_1)\cap (z-1)\bl M_z \otimes (N\otimes M)\br
=(z-1)(M_z\otimes K_1)$.
Thus $(z-1)^{-1}\Rren_{N\otimes M,M_z}\vert_{K_1 \otimes M_z}$ is well defined.
Hence, we have
$\La(K_1,M)\le \La(N\otimes M,M)-1=\La(N,M)-1$.
Hence we have
$\La(K_k/K_{k+1},M)\le\La(K_1,M)<\La(N,M)$ for $k\ge1$
by Proposition~\ref{prop:subqcL}.
\end{proof}

\begin{corollary}\label{cor:compest}
Let $M$ and $N$ be simple modules in $\uqm$.
We assume that one of them is real and   $M \otimes N$ is not simple.
We write
$$[M\tens N]=[M\hconv N]+[ M\sconv N ]+\sum_{k}[S_k]$$
with simple modules $S_k$ in the Grothendieck ring $K(\uqm)$.
Then we have
\begin{enumerate}[{\rm(i)}]
\item
If $M$ is real, then we have
$\La(M,  M\sconv N)<\La(M,N)$, $\La(M\hconv N,M)<\La(N,M)$ and
$\La(M, S_k)<\La(M,N)$, $\La(S_k,M)<\La(N,M)$.
\item
If $N$ is real, then we have
$\La(N, M\hconv N)<\La(N,M)$, $\La( M \sconv N, N)<\La( M,N)$ and
$\La(N, S_k)<\La(N,M)$, $\La(S_k,N)<\La(M,N)$.
\end{enumerate}
\end{corollary}

The following theorem is a $\uqpg$-analogue of \cite[Theorem 4.1]{KK18}:

\begin{theorem} \label{thm:commcv}
Let $X$ be a simple module and $M$  a real simple module in $\uqm$.
If $[X] = [M] \phi$   for some $\phi$ in $K(\uqm)$,
then $X \simeq M\otimes  Y$ for some simple module $Y$ in $\uqm$ which strongly commutes with $M$.
\end{theorem}

\begin{proof}
We may assume that
$$\phi=\sum_{i \in K }[Y_i]-\sum_{j\in K'}[Z_j],$$
where $Y_i$ and $Z_j$ are simple modules in $\uqm$ and there is no pair $(i,j)\in K\times K'$  such that $Y_i\simeq Z_j$.
It follows that
$$[X] +\sum_{j\in K' }[M\otimes  Z_j]= \sum_{i \in K}[M\otimes  Y_i] \qt{ and } [X] +\sum_{j\in K' }[Z_j\otimes  M]= \sum_{i \in K}[Y_i\otimes  M] $$
in $K(\uqm)$.
Take $i_0$ such that $\La(M,Y_{i_0})=\max\set{\La(M,Y_i)}{i\in K}$.
For any $j \in K'$, the head $M \hconv Z_j$ appears as a subquotient of  some $M \otimes  Y_i$.
Since $M \hconv Z_j \not \simeq  M \hconv Y_i$, we have
$$
\La(M,Z_j)=\La(M,M\hconv Z_j) < \La(M,M \hconv Y_i) = \La(M,Y_i) \le  \La(M,Y_{i_0}).
$$
Since any simple  subquotient $S$ of $M\otimes  Z_j$ satisfies
$$ \La(M,S) \le \La(M,Z_j) <\La(M,Y_{i_0}) =\La(M,M\hconv Y_{i_0}),$$
we conclude that $M\hconv Y_{i_0}$ does not appear in $M\otimes  Z_j$ for any $j \in K'$.
Hence $$X\simeq M\hconv Y_{i_0}.$$

In particular, we have
$$\La(M, Y_i)\le \La(M,Y_{i_0})=\La(M,M\hconv Y_{i_0})=\La(M, X)
\qt{for any $i\in K$.}$$

Take $i_1$ such that $\La(Y_{i_1},M)=\max\set{\La(Y_i,M)}{i\in K}$.
For any $j \in K'$, the head $Z_j \hconv M$ appears as a subquotient of some $Y_i \otimes M$. Since $Z_j \hconv M \not\simeq Y_i \hconv M$, we have
$$
\La(Z_j,M)=\La(Z_j\hconv M,M) < \La(Y_i\hconv M,M) = \La(M,Y_i) \le  \La(Y_{i_1},M).
$$
Thus, by the same reasoning as above, we have
$$  X \simeq Y_{i_1} \hconv M \simeq M \sconv Y_{i_1}.$$

In particular, we have
$$\La(Y_{i_0},M)
\le \La(Y_{i_1},M)=\La(Y_{i_1}\hconv M,M)=\La(X, M)=\La(M \hconv Y_{i_0},M). $$
Hence, if $M$ and $Y_{i_0}$ do not strongly commute, the
inequality $\La(Y_{i_0},M) \le \La(M \hconv Y_{i_0},M)$ contradicts Corollary~\ref{cor:compest} (i). Thus $M$ and $Y_{i_0}$  strongly commute and hence
$X\simeq\M\tens Y_{i_0}$.
\end{proof}

\begin{definition} [{cf.\ \cite[Definition 2.5]{KK18}}]
A sequence $(L_1,\ldots,L_r)$ of real simple modules  in $\uqm$ is called a \emph{normal sequence} if the composition of the
$R$-matrices
\eqn
\rmat{L_1,\ldots,L_r}\seteq
\displaystyle\prod_{1\le i <k \le r} \rmat{L_i,L_k} =&&(\rmat{L_{r-1},L_r})  \circ \cdots \circ (\rmat{L_2,L_r}\circ \cdots \circ \rmat{L_2,L_3})  \circ (\rmat{L_1,L_r} \circ \cdots  \circ \rmat{L_1,L_2})
\\
  &&: L_1\tens \cdots \tens L_r \longrightarrow L_r \tens \cdots \tens  L_1
\eneqn
does not vanish.
\end{definition}

The following two lemmas can be proved by the same arguments in \cite[Section 2.3]{KK18} with $\La$.

\begin{lemma}
If $(L_1,\ldots,L_r)$ is a normal sequence of real simple modules in $\uqm$, then
the image of $\rmat{L_1,\ldots,L_r}$ is simple and coincides with
the head of $L_1\conv \cdots \conv L_r$
and also with the socle of $ L_r \conv \cdots \conv  L_1$,
\end{lemma}

\begin{lemma} \label{lemma:normal1}
Let $(L_1,\ldots,L_r)$  be a sequence of real simple  modules in $\uqm$.
Then the following three conditions are equivalent:
\bna
\item $(L_1,\ldots,L_r)$ is a normal sequence,
\item
$(L_2,\ldots,L_r)$ is a normal sequence and
$$\La(L_1, \hd(L_2\tens\cdots \tens L_r)) = \sum\nolimits_{2\le j\le r} \La(L_1,L_j),$$
\item
$(L_1,\ldots,L_{r-1})$ is a normal sequence and
$$\La(\hd(L_1\tens\cdots \tens L_{r-1}), L_r) = \sum\nolimits_{1\le j\le r-1} \La(L_j,L_r).$$
\ee
\end{lemma}

\begin{lemma}
For real simple modules $L$, $M$ and $N$ in $\uqm$, $(L,M,N)$ is a normal sequence if either $L$ and $M$ strongly commute or
$L$ and $N^*$ strongly commute.
\end{lemma}

\begin{proof}
The first case follows from Lemma~\ref{lem:LMN1},
and the second case from Corollary~\ref{cor:aadd}.
\end{proof}

\begin{corollary}
For real simple modules $L$, $M$ and $N$ in $\uqm$, $(L^*,M,N)$ is a normal sequence if and only if $(M,N,L)$ is a normal sequence.
\end{corollary}

\begin{proof}
Proposition~\ref{prop: sym lam} implies that
$$\La(L^*, M)+\La(L^*,N)-\La(L^*,M\hconv N)
=\La(M,L)+\La(N,L)-\La(M\hconv N.L).$$
Then  our assertion follows from Lemma~\ref{lemma:normal1}
since $(M,N)$ is a normal sequence.
\end{proof}

\vs{5ex}

\section{Cluster algebras} \label{sec: cluster algebra}
In this section, we briefly recall the definition of cluster algebra  with little modifications.
For more detail, we refer the reader to \cite{BZ05,FZ02}.
Fix a countable index set $K=K^\ex \sqcup K^\fr$ which decomposes into subset $K^\ex$ of exchangeable indices  and  a  subset $K^\fr$ of frozen indices.

Let $\widetilde{B}=(b_{ij})_{(i,j)\in K \times \Kex}$ be an integer-valued matrix such that
\eq &&\hs{2ex}
\parbox{75ex}{
\begin{enumerate}[{\rm (a)}]
\item  for each $j \in \Kex$, there exist finitely many $i \in \K$ such that $b_{ij} \ne 0$,
\item  the {\it principal part} $B \seteq (b_{ij})_{i,j \in \Kex}$ is skew-symmetric.
\end{enumerate}
}\label{eq: condition B}
\eneq
We extend the definition of $b_{ij}$ for $(i,j)\in K\times K$ by:
$$\text{
$b_{ij}=-b_{ji}$ if $i\in \Kex$ and $j\in K$ and
$b_{ij}=0$ for $i,j\in \Kfr$, }$$
so that $(b_{ij})_{i,j\in \K}$ is skew-symmetric.

To the matrix $\tB$, we associate the quiver $\mathfrak{Q}_{\tB}$ such that
the set of vertices is $\K$ and
the number of arrows from $i\in\K$ to $ j\in\K$ is $\max(0,b_{ij})$.
Then, $\mathfrak{Q}_{\tB}$ satisfies that
\begin{eqnarray} &&
\parbox{75ex}{
\begin{enumerate}[{\rm (a)}]
\item the set of vertices of $\mathfrak{Q}_{\tB}$ are labeled by $\K$,
\item $\mathfrak{Q}_{\tB}$ does not have any loop, any $2$-cycle nor arrow between frozen vertices,
\item each  exchangeable vertex $v$ of $\mathfrak{Q}_{\tB}$ has {\it finite degree}; that is, the number of arrows incident with $v$ is finite.
\end{enumerate}
}\label{eq: Quiver condition}
\end{eqnarray}

Conversely, for a given quiver satisfying~\eqref{eq: Quiver condition}, we can associate a matrix $\tB$ by
\begin{align}\label{eq: bij}
b_{ij} \seteq \text{(the number of arrows from $i$  to $j$)} \hspace{-.2ex}   -  \hspace{-.2ex}  \text{(the number of arrows from $j$  to $i$)}.
\end{align}
Then $\tB$ satisfies~\eqref{eq: condition B}.

Let $L=(\la_{ij})_{i,j\in K}$ be a skew-symmetric integer-valued $K \times K$-matrix. We say that $L$ is \emph{compatible} with
$\widetilde{B}$ with a positive integer $d \in \Z_{\ge1}$, if
$$   \sum_{k \in \K} \lambda_{ik}b_{kj} =\delta_{i,j}d \qquad \text{ for each $i \in \K$ and $j \in \Kex$}.$$

Let $\{ X_i\}$ be the set of mutually commuting indeterminates.

\Def
For a commutative ring $A$, we say that a triple $\Seed_{\Uplambda}=(\{ x_i \}_{i \in K}, L,\tB) $ is a \emph{$\Uplambda$-seed} in $A$ if
\bna
\item
there exists an injective algebra homomorphism $\Z[X_i]_{i \in K}$ into $A$ such that $X_i \mapsto x_i$,
\item $(L,\tB)$ is a compatible pair with respect to $d \in \Z_{\ge 1}$.
{\em In this paper, we always assume that $d=2$.}
\ee
\edf

For a $\Uplambda$-seed $\Seed_{\Uplambda}=(\{ x_i \}_{i \in K}, L,\tB)$,
we call the set $\{ x_i \}_{i \in K}$  the \emph{cluster} of $\Seed_{\Uplambda}$ and
its elements the \emph{cluster variables}.
An element of the form   $x^{{\bf a}}$ $\bigl({\bf a} \in \Z_{\ge 0}^{\oplus \K}\bigr)$
is called a \emph{cluster monomial},
where
$$ x^{{\bf c}}  \seteq  \prod_{k \in K} x_{i_k}^{c_{i_k}} \quad \text{ for }  \ {\bf c}=(c_{i})_{i\in\K} \in \Z^{\oplus \K}.$$

Let $\Seed_{\Uplambda}=(\{ x_i \}_{i \in K}, L,\tB)$ be a $\Uplambda$-seed in
a field $\mathfrak{K}$ of characteristic $0$.
For each $k \in \Kex$, we define
\begin{eqnarray} &&
\parbox{81ex}{
\begin{enumerate}
\item[{\rm (a)}] $\mu_k(L)_{ij} =
\begin{cases}
  -\la_{kj}+\displaystyle\sum _{t\in\K} \max(0, -b_{tk}) \la_{tj} \quad \  & \text{if} \ i=k, \ j\neq k, \\
  -\la_{ik}+\displaystyle\sum _{t\in\K} \max(0, -b_{tk}) \la_{it} & \text{if} \ i \neq k, \ j= k, \\
   \la_{ij} & \text{otherwise,}
\end{cases}$

\vs{2ex}
\item[{\rm (b)}]  $\mu_k(\tB)_{ij} =
\begin{cases}
  -b_{ij} & \text{if}  \ i=k \ \text{or} \ j=k, \\
  b_{ij} + (-1)^{\true(b_{ik} < 0)} \max(b_{ik} b_{kj}, 0) & \text{otherwise,}
\end{cases}
$

\vs{1ex}
\item[{\rm (c)}] $ \hspace{0.25em} \    \mu_k(x)_i  =\begin{cases}
x^{{\bf a}'}  +   x^{{\bf a}''}, & \text{if} \ i=k, \\
x_i & \text{if} \ i\neq k,
\end{cases}
$
\end{enumerate}
}\label{eq: mutation in a direction k}
\end{eqnarray}
where ${\bf a}'\seteq(a_i')_{i\in\K}$ and ${\bf a}''\seteq(a_i'')_{i\in\K} \in \Z^{\oplus \K}$ are defined as follows:
\begin{align*}
a_i'= \begin{cases}
  -1 & \text{if} \ i=k, \\
 \max(0,b_{ik}) & \text{if} \ i\neq k,
\end{cases} \qquad
a_i''= \begin{cases}
  -1 & \text{if} \ i=k, \\
 \max(0,-b_{ik}) & \text{if} \ i\neq k.
\end{cases}
\end{align*}
Then the triple
$$\mu_k(\Seed_{\Uplambda}) \seteq (   \{ \mu_k(x)_i\}_{k \in K}, \mu_k(L),\mu_k(\tB) )$$
becomes a new $\Uplambda$-seed in $\mathfrak{K}$ and we call it the \emph{mutation} of $\Seed_{\Uplambda}$ at $k$.

The \emph{cluster algebra $\mathscr{A}(\Seed_{\Uplambda})$ associated to the $\Uplambda$-seed} $\Seed_\Uplambda$
 is the $\Z$ -subalgebra of the field $\mathfrak{K}$  generated by all the cluster variables in the $\Uplambda$-seeds obtained from
$\Seed_{\Uplambda}$ by all possible successive mutations.

A \emph{cluster algebra structure associated to a $\Uplambda$-seed $\Seed_\Uplambda$} on a $\Z$-algebra $A$
is a family $\mathscr{F}$ of $\Uplambda$-seeds in $A$ such that
\begin{enumerate}[{\rm (a)}]
\item for any $\Uplambda$-seed $\Seed_{\Uplambda}$ in $\mathscr{F}$,
the cluster algebra $\mathscr{A}(\Seed_\Uplambda)$ is isomorphic to $A$,
\item any mutation of a $\Uplambda$-seed in $\mathscr{F}$ is in $\mathscr{F}$,
\item for any pair $\Seed_\Uplambda$, $\Seed'_\Uplambda$ of $\Uplambda$-seeds in $\mathscr{F}$,
$\Seed'_\Uplambda$ can be obtained from $\Seed_\Uplambda$ by
a finite sequence of mutations.
\end{enumerate}

Note that the definition of cluster algebra associated to a $\Uplambda$-seed is designed for the
Grothendieck ring $\K(\Ca_\g)$ of $\Ca_\g$
and can be understood as an intermediate one between a cluster algebra and a quantum cluster algebra.
When we ignore $L$ in each $\Uplambda$-seed $\Seed_\Uplambda$, we recover the definition of cluster algebra.

\section{Monoidal categorification} \label{sec: monoidal categorification}
In this section, we construct a $\uqpg$-analogue of \cite[Section 7]{KKKO18}.
{\em From now on, $\shc$ is  a full subcategory of $\uqm$
stable under taking tensor products, subquotients and extensions.}
Note that $K(\shc)$ has a $\Z$-basis consisting of
the isomorphism classes of simple modules.

\begin{definition}
A \emph{monoidal seed in $\shc$} is
a pair $\seed = (\{ M_i\}_{i\in \K },\widetilde B)$
consisting of
a strongly
commuting family $\{ M_i\}_{i\in\K}$ of real simple modules in
$\shc$ and
an  integer-valued $\K\times\Kex$-matrix
$\widetilde B = (b_{ij})_{(i,j)\in\K\times\Kex}$
satisfying the conditions in~\eqref{eq: condition B}.

For $i\in\K$, we call $M_i$ the $i$-th {\em cluster variable module} of $\seed$.
\end{definition}

For a monoidal seed  $\seed=(\{M_i\}_{i\in\K}, \widetilde B)$, let
$\La^\seed=(\La^\seed_{ij})_{i,j\in\K}$
be the skew-symmetric matrix
given by $\La^\seed_{ij}=\Lambda(M_i,M_j)$.

\begin{definition}
For $k\in\Kex$, we say that a  monoidal seed  $\seed = (\{ M_i\}_{i\in \K },\widetilde B)$ \emph{admits a mutation in direction $k$} if
there exists a simple object  $M_k' \in \shc$ such that
\bna
  \item
there exist exact sequences in $\shc$
\eqn
&&0 \to  \sotimes_{b_{ik} >0} M_i^{\tens b_{ik}} \to M_k \tens M_k' \to
 \sotimes_{b_{ik} <0} M_i^{\tens (-b_{ik})} \to 0, \label{eq:ses_mutation1}\\
 &&0 \to  \sotimes_{b_{ik} <0} M_i^{\tens(-b_{ik})} \to M_k' \tens M_k \to
  \sotimes_{b_{ik} >0} M_i^{\tens b_{ik}} \to 0.\label{eq:ses_mutation2}
\eneqn
\item
The pair $\mu_k(\seed)\seteq
(\{M_i\}_{i\neq k}\cup\{M_k'\},\mu_k(\widetilde B))$ is
a monoidal seed in $\shc$.
\end{enumerate}
\end{definition}
Condition (b) is equivalent to saying that $M'_k$ is real and strongly commuting with $M_i$ for any $i\in\K\setminus\st{k}$.
\smallskip

\begin{definition} \label{def:admissible}
A monoidal seed $\seed=(\{M_i\}_{i\in\K}, \widetilde B)$ is called \emph{admissible} if,
for each $k\in\Kex$, there exists a simple object $M'_k$ of $\shc$  such that
 there is an exact sequence in $\shc$
\begin{align}\label{eq:ses mutation}
0 \to  \sotimes_{b_{ik} >0} M_i^{\tens  b_{ik}} \to M_k \otimes M_k' \to \sotimes_{b_{ik} <0} M_i^{\tens  (-b_{ik})} \to 0,
\end{align}
and $M_k'$ commutes with $M_i$  for any  $i \neq k$.
\end{definition}

Note that $M'_k$ is uniquely determined by $k$ and $\seed$.
Indeed, it follows from
$ M_k\hconv M'_k\simeq\sotimes_{b_{ik} <0} M_i^{\tens(-b_{ik})}$ and \cite[Corollary 3.7]{KKKO15}.

It is evident that  a monoidal seed which admits a mutation
at all $k\in\Kex$ is admissible.
Indeed, the converse is true.

\begin{proposition} \label{prop:condition simplified}
Let $\seed=(\{M_i\}_{i\in\K},\widetilde B)$ be an admissible monoidal seed in $\shc$ and $k\in\Kex$,
Let $M'_k$ be as in {\rm Definition~\ref{def:admissible}}.
Then we have the following properties.
\bnum
\item
The monoidal seed $\seed$ admits a mutation in direction $k$.
In particular, $M_k'$ is a real simple object. \label{item:1}
\item For any $j \in \K$, we have $(\Lambda^\seed\;\widetilde{B})_{jk}=-2 \delta_{jk} \de(M_k,M_k')$. \label{item:2}
\item For any $j\in\K$, we have
\begin{align}\label{eq:Larel}
&\La(M_j,M'_k)=-\La(M_j,M_k)-\sum_{b_{ik}<0}\La(M_j,M_i)b_{ik},\\[1ex]
&\La(M'_k,M_j)=-\La(M_k,M_j)+\sum_{b_{ik}>0}\La(M_i,M_j)b_{ik}.
\end{align} \label{item:Larel}
\end{enumerate}
\end{proposition}

\begin{proof}
(i) The reality of $M'_k$ follows from the exact sequence
\eqref{eq:ses mutation} by applying Proposition \ref{prop:real}~\eqref{item: real ii} to the case
$$M=M_k, \ N=M'_k, \ X= \sotimes_{b_{ik} > 0} M_i^{\tens b_{ik}} \ \text{and} \  Y= \sotimes_{b_{ik} < 0} M_i^{\tens(-b_{ik})}. $$
Note that $N\tens M$ has the same length as the one of $M\tens N$, that is $2$.
Since $N\sconv M\simeq M\hconv N\simeq Y$ and
$N\hconv M\simeq M\sconv N\simeq X$,
we have an exact sequence $0\to Y\to N\tens M\to X\to0$.

\medskip
\noindent \eqref{item:Larel} follows from
\begin{align*}
\La(M_j,M_k)+\La(M_j,M'_k)&=\La(M_j,M_k\hconv M'_k)
=\La\bl M_j,\sotimes_{b_{ik} <0} M_i^{\tens(-b_{ik})}\br\\
&=\sum_{b_{ik}<0}\La(M_j,M_i)(-b_{ik})
\end{align*}
and
\begin{align*}
\La(M_k,M_j)+\La(M'_k,M_j)&=\La(M'_k\hconv M_k,M_j)
=\La\bl\sotimes_{b_{ik} >0} M_i^{\tens b_{ik}}, M_j\br\\*
&=\sum_{b_{ik}>0}\La(M_i,M_j)b_{ik}.
\end{align*}

\noindent \eqref{item:2} follows from \eqref{item:Larel} as follows:
\begin{align*}
2 \delta_{jk} \de(M_k,M_k')&=2\de(M_j,M'_k)=\La(M_j,M'_k)+\La(M'_k,M_j)\\
 &=-\La(M_j,M_k)-\La(M_k,M_j)-\sum_{b_{ik}<0}\La(M_j,M_i)b_{ik}
+\sum_{b_{ik}>0}\La(M_i,M_j)b_{ik}\\
&=-2\de(M_j,M_k)-\sum_{b_{ik}<0}\La(M_j,M_i)b_{ik}-\sum_{b_{ik}>0}\La(M_j,M_i)b_{ik}\\
&=-\sum_{i\in\K}\La(M_j,M_i)b_{ik}. \qedhere
\end{align*}
\end{proof}

\Def
We say that a monoidal seed $\seed=(\{M_i\}_{i\in \K},\widetilde B)$ is $\Uplambda$-admissible if $\seed$ is an admissible
monoidal seed and $(-\Lambda^\seed,\widetilde B)$ is compatible with $2$; i.e.,
$$ (\Lambda^\seed\tB)_{jk}=-2\delta_{jk}\qt{for $(j,k)\in\K\times\Kex$.}$$
\edf
Note that the compatibility condition  is equivalent to saying that
$\de(M_k,M'_k)=1$
for any $k\in\Kex$.

\smallskip
For a monoidal seed
$\seed=(\{M_i\}_{i\in \K},\widetilde B)$ in $\shc$,
we define the triple $[\seed]_\Uplambda$ in $K(\shc)$ by
$$[\seed]_\Uplambda\seteq\bl\{ [M_i] \}_{i \in K},-\La^\seed,\tB\br.$$

If $\seed$ is a $\Uplambda$-admissible monoidal seed, then $[\seed]_\Uplambda$
is a $\Uplambda$-seed.

The following lemma immediately follows from
Proposition~\ref{prop:condition simplified}.
\Lemma
Let $\seed=(\{ M_i \}_{i \in K},\tB)$ be a  $\Uplambda$-admissible monoidal seed,
and $k\in \Kex$.
Then we have
$$\mu_k\bl[\seed\,]_\Uplambda\br=[\,\mu_k(\seed)\,]_\Uplambda.$$
In particular, $\bl-\La^{\mu_k(\seed)},\mu_k(\tB)\br$ is compatible with $2$.
\enlemma

\begin{definition} 
A category $\shc$ is called a \emph{$\Uplambda$-monoidal categorification} of a cluster algebra $A$ if
\begin{enumerate}
\item the Grothendieck ring $K(\shc)$ is isomorphic to $A$,
\item there exists a monoidal seed $\seed=(\{ M_i \}_{i \in K},\tB)$ in $\shc$ such that
$$[\seed]_{\Uplambda}\seteq(\{ [M_i] \}_{i \in K},-\La^\seed,\tB)$$
is the initial $\Uplambda$-seed of $A$ and $\seed$ admits successive mutations in all possible directions.
\end{enumerate}
\end{definition}

\begin{definition} \label{def: max com}
A family of real simple modules $\{M_i\}_{i\in K}$ in $\shc$ is called a {\it maximal real commuting family in $\shc$} if it satisfies:
\bna
\item $\{M_i\}_{i\in K}$ are mutually strongly commuting,  and
\item
if a simple  module $X$ strongly commutes with all the $M_i$'s,
then $X$ is isomorphic to $\tens_{i \in \K}M^{\tens a_i}$
for some $\mathbf{a} \in \Z_{\ge 0}^{\oplus K}$.
\ee
\end{definition}

The following theorem can be proved similarly to
 its quiver Hecke algebra version \cite[Theorem 2.4]{KK18}.
\begin{theorem}\label{th;main}
Let $X$ be a simple module in $\shc$ which is a monoidal categorification of a cluster algebra $A$ associated to a $\Uplambda$-seed. If
 $\seed=(\{M_i\}_{i \in K}, \tB)$ a monoidal seed of $\shc$ and induces a $\Uplambda$-seed of $A$, then $\{M_i\}_{i \in K}$ is a maximal real commuting family in $\shc$.
\end{theorem}

After the introduction of new invariants for pairs of modules in $\uqm$,
the following theorem can be proved similarly
to the one in \cite[Theorem 7.1.3]{KKKO18} with a small modification.
Since it is one of the principal results of this paper, we repeat its proof.

\begin{theorem}\label{th:main}
Let $\seed=(\{M_i\}_{i\in K},\widetilde B)$
be a $\Uplambda$-admissible monoidal seed  in $\shc$, and set
$$[\seed]_\Uplambda\seteq(\{ [M_i]\}_{i\in\K}, -\Lambda^\seed,\widetilde B).$$
We assume that
\begin{align}\label{eq:Cluster}
\hs{-15ex}\text{the algebra $K(\shc)$ is isomorphic to $\mathscr{A}([\seed]_\Uplambda)$}.
\end{align}
Then, for each $x\in\Kex$, the monoidal seed
$\mu_x(\seed)$
 is $\Uplambda$-admissible in $\shc$.
\end{theorem}

\begin{proof}
Set $N_i\seteq\mu_x(M)_i$ and $b'_{ij}\seteq \mu_x(\tB)_{ij}$ for $i \in \K$ and $j \in \Kex$, i.e.
$$\mu_x(\seed)=\bl\st{N_i}_{i\in \K}, (b'_{i,j})_{(i,j)\in\K\times\Kex}\br.$$
By Definition~\ref{def:admissible}, it is enough to show that,
for any $y\in\Kex$,
there exists a  simple module $M_y'' \in \shc$ such that
there is a short exact sequence
\begin{align}\label{eq:seqdes}
\xymatrix{
0 \ar[r] &  \sotimes_{b'_{iy} > 0} N_i^{\tens b'_{iy}} \ar[r] & N_y \tens  M_y'' \ar[r] & \sotimes_{b'_{iy} <0} N_i^{\tens (-b'_{iy})}   \ar[r] & 0
}
\end{align}
and
$$\de(N_i,M_y'')=0 \quad \text{for $i \neq y$.}$$

If $x=y$, then $b'_{iy}=-b_{ix}$ and hence $M''_y=M_x$ satisfies the desired condition.

Assume that $x\not=y$ and $b_{xy}=0$. Then $b'_{iy} = b_{iy}$ for any $i$
and $N_i=M_i$ for any $i\not=x$.
Hence  $M_y''=\mu_y(M)_y$ satisfies the desired condition.

\medskip
We will show the assertion in the case $b_{xy} > 0$.
We omit the proof of the case $b_{xy} <0$ because it
can be shown in a similar way.

Recall that we have
\begin{align} \label{eq:bxy>0 mutation}
b_{iy}'=\begin{cases}
b_{iy}+b_{ix}b_{xy} & \text{if } b_{ix} >0, \\
b_{iy} & \text{if } b_{ix} \le 0
\end{cases}
\end{align}
for $i \in \K$ different from $x$ and $y$.

Set
\begin{align*}
&M_x'\seteq\mu_x(M)_x, \hs{3ex} M_y'\seteq\mu_y(M)_y, \\
&C\seteq\sotimes_{b_{ix} >0} M_i^{\tens b_{ix}}, \quad  S\seteq\sotimes_{b_{ix} <0, \ i\neq y} M_i^{\tens -b_{ix}}, \\
&P\seteq\sotimes_{b_{iy} >0, i\neq x} M_i^{\tens b_{iy}}, \quad  Q:=\sotimes _{b_{iy}' <0, \ i\neq x} M_i^{\tens-b_{iy}'}, \\
&A\seteq\sotimes_{b_{iy}' \le 0, \ b_{ix} >0} M_i^{\tens b_{ix}b_{xy}} \tens \sotimes_{\substack{b_{iy} <0,\; b_{iy}' >0,\;b_{ix} >0}} M_i^{\tens -b_{iy}}
 \\*
&\hspace{3ex}\simeq
\sotimes_{\substack{b_{iy} <0,\; b_{ix} >0}}
M_i^{\tens \min(b_{ix}b_{xy},-b_{iy})},\\
&B\seteq\sotimes_{ b_{iy} \ge 0,  \ b_{ix} >0} M_i^{\tens b_{ix}b_{xy}} \tens \sotimes_{\substack{b_{iy}' >0, \; b_{iy} <0,\;b_{ix} >0}} M_i^{\tens b_{iy}'}.
\end{align*}

Set
 \eqn
&&L\seteq(M_x')^{\tens b_{xy}},   \quad
V\seteq M_x^{\tens b_{xy}}
\eneqn
 and
\eqn
X \seteq\sotimes_{b_{iy} >0} M_i^{\tens b_{iy}}\simeq
M_x^{\tens b_{xy}} \tens P =V \tens P,
 \hs{6ex} Y\seteq\sotimes_{b_{iy} <0} M_i^{\tens-b_{iy}}\simeq Q \tens A.
\eneqn

Then \eqref{eq:seqdes} reads  as
\eq
\xymatrix{
0 \ar[r] &   B \tens P \ar[r]  & M_y \tens M''_y \ar[r] & L  \tens Q\ar[r] & 0.
}\label{eq:seqdes'}
\eneq

Note that we have
\eq
&&0 \to C  \to M_x \tens M_x' \to  M_y^{\tens b_{xy}} \tens S \to 0,  \\
&&0 \to X \to  M_y \tens M_y' \to  Y \to 0. \label{eq:sesMy}
\eneq

Taking the tensor products of $L=(M_x')^{\tens b_{xy}}$  and \eqref{eq:sesMy}, we obtain
\eqn
&&
\xymatrix@R=.5ex{0 \ar[r]& L \tens X  \ar[r]&
  L \tens (M_y \tens M_y')\ar[r]& L \tens Y \ar[r]& 0, \\
0\ar[r]&  X \tens L \ar[r]&
(M_y \tens M_y') \tens L \ar[r]&  Y\tens L \ar[r]& 0.}
\eneqn

Since $L$ commutes with  $M_y$,
we have
\eqn
&&\Lambda(L, Y)=\Lambda(L, M_y \hconv M_y') \\
 &&=\Lambda(L, M_y) +\Lambda(L, M_y') = \Lambda(L,M_y \tens M_y' ).
\eneqn

On the other hand, we have
\begin{align*}
&\Lambda(M_x',X) - \Lambda(M_x',Y)  \allowdisplaybreaks\\
&=\Lambda(M_x',\sotimes_{b_{iy} >0} M_i^{\tens b_{iy}})-\Lambda(M_x',\sotimes_{b_{iy} <0} M_i^{\tens -b_{iy}})  \allowdisplaybreaks\\
&=\sum_{b_{iy} >0} \Lambda(M_x',M_i)b_{iy} -\sum_{b_{iy} <0} \Lambda(M_x',M_i)(-b_{iy})  \allowdisplaybreaks\\
&=\sum_{i\in\K} \Lambda(M_x', M_i)b_{iy}
=\sum_{i \neq x} \Lambda(M_x',M_i)b_{iy} + \Lambda(M_x',M_x)b_{xy} \allowdisplaybreaks\\
&=\sum_{i \neq x} \Lambda(M_x', M_i)(b_{iy}' -\delta(b_{ix} >0)b_{ix} b_{xy}) + \Lambda(M_x',M_x)b_{xy} \allowdisplaybreaks\\
&=\sum_{i \neq x} \Lambda(M_x', M_i) b_{iy}'
-\sum_{b_{ix} >0} \Lambda(M_x', M_i)b_{ix}b_{xy} + \Lambda(M_x',M_x)b_{xy} \allowdisplaybreaks\\
&\mathop=\limits_{\mathrm(\star)}0-\Lambda(M_x',\sotimes_{b_{ix} >0} M_i^{\tens b_{ix}})b_{xy} + \Lambda(M_x',M_x)b_{xy} \allowdisplaybreaks\\
&=\bl-\Lambda(M_x',\sotimes_{b_{ix} >0} M_i^{\tens b_{ix}} ) + \Lambda(M_x',M_x) \br b_{xy} \allowdisplaybreaks\\
&=(-\Lambda(M_x', M_x' \hconv M_x) + \Lambda(M_x',M_x) )b_{xy} \allowdisplaybreaks\\
&=(-\Lambda(M_x', M_x')-\Lambda(M_x', M_x) + \Lambda(M_x',M_x) )b_{xy}=0.
\end{align*}
Note that we have used the compatibility of
$\bl \Lambda(\mu_x(M)_i,\mu(M)_j)\br_{i,j}$ and $\mu_x(B)$
when we derive the equality ($\star$).

Since $L=(M'_x)^{\tens b_{xy}}$, the equality $\La(M'_x,X)=\La(M'_x,Y)$ implies
\eqn
&& \Lambda(L,X)
=\Lambda(L,Y)=\La(L, M_y \tens M_y'). \label{eq:LaLX=LaLY}
\eneqn
Hence  the following  diagram is commutative by Proposition \ref{prop:inequalities}~\eqref{it:commr}:
\eqn \label{eq:commutative ses}
&&\ba{c}\xymatrix{
0 \ar[r] &  L \tens X \ar[r] \ar[d]^{\rmat{L,X}} & L \tens (M_y \tens M_y')
 \ar[d]^{\rmat{L,M_y \tens M_y'}} \ar[r] &  L \tens Y \ar[r] \ar[d]_{\rmat{L,Y}} ^{\bwr}&0\\
0 \ar[r] &  X \tens L \ar[r]  &  (M_y \tens M_y') \tens L \ar[r] & \, Y \tens L \ar[r] & 0.
}\ea
\eneqn
Note that  since $L=(M_x')^{\tens b_{xy}}$ commutes with $Q$ and $A$, $\rmat{L,Y}$ is an isomorphism.  Hence
we have
$$\Im(\rmat{L,Y}) \simeq L \tens Y.$$
 Therefore we obtain  an exact sequence
\eq
\xymatrix{0\ar[r]&\Im(\rmat{L,X})\ar[r]&\Im(\rmat{L, M_y\tens
 M'_y})\ar[r]
&L\circ Y\ar[r]&0.}\label{eq:exIm}
\eneq

On the other hand, $\rmat{L, M_y \tens M_y'}$ decomposes by Proposition~\ref{prop:rcomp} as follows:
\eqn
\xymatrix@C=6em{
L \tens M_y \tens M_y'  \ar[r]^\sim_{\rmat{L,M_y} \tens M_y'}
\ar@/^2pc/ [rr]^{\rmat{L, M_y \tens M_y'}}&
 M_y \tens L \tens M_y'  \ar[r]_{M_y \tens \rmat{L,M_y'}} &
 M_y  \tens M_y' \tens L.
 }
 \eneqn
Since $L=(M_x')^{\tens b_{xy}}$ commutes with $M_y$,
the homomorphisms $\rmat{L,M_y} \tens M_y'$ is an isomorphism and hence we have
\eqn
{\rm Im}
 (\rmat{L, M_y \tens M_y'}) \simeq  M_y \tens (L\hconv M'_y).
 \eneqn

Similarly, $\rmat{L, X}$ decomposes as follows:
\eqn
\xymatrix@C=6em{
L \tens V \tens P  \ar[r]_{\rmat{L,V} \tens  P}
\ar@/^1.5pc/ [rr]^{\rmat{L,X}}&
 V \tens L \tens P  \ar[r]_{ V\tens \rmat{L,P}}^\sim &
V \tens P  \tens L.
 }
 \eneqn
Since $L$ commutes with $P$,
the  homomorphism  $V \tens \rmat{L,P} $ is an isomorphism and hence we have
  \eqn
{\rm Im}
 (\rmat{L,X})
  \simeq (L \hconv V) \tens P
\simeq \bl(M_x')^{\tens b_{xy}} \hconv M_x^{\tens b_{xy}}\br \tens P.
 \eneqn
On the other hand,
Lemma~\ref{lem:MN} implies that
\eqn
(M_x')^{\tens b_{xy}} \hconv M_x^{\tens b_{xy}}
\simeq (M_x' \hconv M_x)^{\tens b_{xy}}
\simeq  C^{\tens b_{xy}}
\simeq B \tens A,
\eneqn
and hence we obtain
  \eqn
{\rm Im}
 (\rmat{L,X})
\simeq (B \tens P)  \tens A.
\eneqn
Thus the exact sequence \eqref{eq:exIm} becomes the exact sequence in $\shc$:
\eq
&&\xymatrix{
0 \ar[r] &  (B \tens P) \tens A \ar[r]  & M_y \tens (L\hconv M_y') \ar[r] & (L  \tens Q) \tens A \ar[r] & 0.
}\label{eq:exML}
\eneq


Thus we obtain the identity in $K(\uqm)$:
\eqn
[M_y]\, [L \hconv M_y'] = \bl\, [B \tens P] + [L \tens Q]\,\br\, [A].
\eneqn

On the other hand,  the hypothesis \eqref{eq:Cluster} implies that
there exists $\phi \in K(\shc)$
corresponding to $\mu_y\mu_x([M])$ so that
it satisfies
\eq
[M_y] \phi =  [B \tens P] + [L \tens Q].
\label{eq:Mphi}
\eneq

Hence, in $K(\shc)$, we have
\eqn
[M_y] \phi [A] = \bl  [B \tens P] + [L \tens Q]\br [A ]=  [M_y][L \tens M_y'].
\eneqn
Since $K(\shc)$ is an integral domain, we conclude that
\eqn
\phi [A] = [L \hconv M_y'].
\eneqn

By Theorem~\ref{thm:commcv}, there exists a simple module $M_y''$ such that $\phi=[M_y'']$, since $A$ is real simple.

Now \eqref{eq:Mphi}  implies
$$[M_y\tens M''_y] = [B \tens P] + [L \tens Q].$$
Hence there exists an exact sequence
$$0\To W\To\;M_y\tens M''_y\To Z\To 0,$$
where $W=B \tens P$ and $Z=L \tens Q$ or
$W=L \tens Q$ and $Z=B \tens P$.

Since $   \La(M_y,L\otimes Q) -  \La(M_y,B\otimes P) = -\displaystyle \sum_{b'_{iy}} \La(M_y,M_i)b'_{iy} =2\de(M_y,M_y'') >0$, we have
$$0\To B\tens P\To\;M_y\tens M''_y\To L\tens Q \To 0,$$
by Corollary~\ref{cor:compest}.

Now it remains to prove that
\bna
\item $M_y''$ strongly commutes with $M_i$ $(i \ne x,y)$ and $M_x'$,
\item $M_y''$ is real simple.
\ee
 Take $M$ as one of $M_i$ $(i \ne x,y)$ and $M_x'$.
Then $M_y'' \tens M$ is of length less than or equal to $2$, since $M_y \otimes M_y'' \otimes M$ is of length 2:
$$0\To H \tens M \To\;M_y\tens M''_y\tens M  \To G\tens M  \To 0,$$
where $H \seteq B \tens Q\simeq\sotimes_{b'_{iy}>0}M_i^{\tens b'_{iy}}$ and $G \seteq L \tens Q$.

Assume that $M_y'' \tens M$ is of length $2$:
\begin{align}\label{eq: length 2 assumption}
0\To U  \To  M''_y\tens M  \To V  \To 0.
\end{align}
By taking tensor $M_y \otimes$ to~\eqref{eq: length 2 assumption},  we have
$$  [M_y][U] =[H][M]. $$
Note that $[H \otimes M]=[\sotimes_{b'_{iy}>0}M_i^{\tens b'_{iy}}] [M]$.
Since $K(\shc)$ has a cluster algebra structure, the cluster variables $[M_y]$ and $[M]$ are prime (\cite{GLS13}). However, $[M_y]$ does not divide
either $[H]$ nor $[M]$ which contradicts $[M_y][U] =[H][M]$.
Thus we obtain (a).

By (a), $M_y''$ strongly commute with  $L \otimes Q$ and  $B \otimes P$, and hence $M_y''$ is  real simple by Proposition~\ref{prop:real}.
It completes the proof of Theorem~\ref{th;main}.
\end{proof}

\begin{corollary}\label{cor:main}
Let  $\seed=(\{M_i\}_{i\in\K},\widetilde B)$ be a $\Uplambda$-admissible
monoidal seed in $\shc$.
Under the assumption \eqref{eq:Cluster},
$\shc$ is a $\Uplambda$-monoidal categorification of the cluster algebra
$\mathscr A([\seed]_\Uplambda)$.
Furthermore,  the following statements hold{\rm:}
\bnum
  \item The  monoidal seed $\seed=(\{M_i\}_{i\in\K},\widetilde B)$   admits successive mutations in all directions.
  \item Any cluster monomial in $K(\shc)$  is the isomorphism class of a  real simple object in $\shc$.
  \item Any cluster monomial in $K(\shc)$ is a Laurent polynomial of the initial cluster variables with coefficient in $\Z_{\ge0}$.
\item For $k\in\Kex$ and the $k$-th cluster variable module $\widetilde{M}_k$ of a monoidal seed $\widetilde{\seed}$ obtained by successive mutations from the initial monoidal seed
$\seed$,  we have
$$\de(\widetilde{M}_k,\widetilde{M}_k')=1.$$
Here $\widetilde{M}_k'$ is the $k$-th cluster variable module of $\mu_k(\widetilde{\seed})$.
\item  Any monoidal cluster $\{\widetilde{M}_i\}_{i\in\K}$ is a maximal real commuting family.
\end{enumerate}

\end{corollary}

\Rem
In \cite{KKO18,KKOP19B}, we constructed several examples of monoidal subcategories $\shc$ of $\uqm$ such that $\shc$ is a $\Uplambda$-monoidal categorification of $K(\shc)$. In those papers, we employed the method of generalized Schur-Weyl duality functors. Namely, we constructed a monoidal functor $\F\col\shc_{\QHA}\to\shc$
such that
\bna
\item $\shc_\QHA$ is a certain monoidal category related with a quiver Hecke algebra,
\item $\shc_\QHA$ is a monoidal categorification of the cluster algebra $K(\shc_\QHA)$,
\item $\F$ induces an isomorphism $K(\shc_{\QHA})\isoto K(\shc)$.
\ee

Further examples of monoidal categorifications obtained from
Theorem~\ref{th:main} will be given in a forthcoming paper.
\enrem

\end{document}